\documentclass[a4paper,12pt]{article}
\title{Automorphism groups and linearizability of rational Fano conic bundle threefolds}

\author{Shuto Abe\blfootnote{\leftline{{\em 2020 Mathematics Subject Classification.} 14J10, 14J45, 14J50, 14L30, 14M20}\protect\linebreak {\em Key words and phrases.} equivariant birational geometry, Fano threefolds, automorphism groups}}

\date{}
\makeatletter
 \def\@evenhead{\hfil\thepage}%
 \def\@oddhead{\hfil\thepage}%
 \def\@evenfoot{\@empty}%
 \def\@oddfoot{\@empty}%
 \makeatother

\usepackage[margin=30truemm]{geometry}
\usepackage{amsmath}
\usepackage{amssymb,bm}
\usepackage{mathrsfs}
\usepackage{amsthm}
\usepackage{amsfonts}
\usepackage{tikz-cd}
\usepackage{tikz}
\usepackage{amscd}
\usepackage{comment}
\usepackage{color}
\usepackage{xcolor}
\usepackage{textcomp}
\usepackage{hyperref}
\usepackage{pb-diagram}
\usepackage[all]{xy} 
\usepackage{lipsum}
\theoremstyle{definition}
\newtheorem{dfn}{Definition}[section]
\newtheorem{thm}[dfn]{Theorem}
\newtheorem{prop}[dfn]{Proposition}
\newtheorem{cor}[dfn]{Corollary}
\newtheorem{lem}[dfn]{Lemma}

\newtheorem*{property}{Property A}
\newtheorem{example}[dfn]{Example}

\newcommand{\Q}{\mathbb{Q}}
\newcommand{\Z}{\mathbb{Z}}
\newcommand{\C}{\mathbb{C}}

\newcommand{\Line}{\mathbb{P}}

\newcommand{\bD}{\begin{dfn}}
\newcommand{\eD}{\end{dfn}}
\newcommand{\bT}{\begin{thm}}
\newcommand{\eT}{\end{thm}}

\renewcommand{\>}{\rangle}
\DeclareMathOperator{\ShO}{\mathscr{O}}

\DeclareMathOperator{\Cl}{Cl}

\DeclareMathOperator{\GL}{GL}

\DeclareMathOperator{\Pic}{Pic}

\DeclareMathOperator{\id}{id}
\DeclareMathOperator{\Aut}{Aut}

\DeclareMathOperator{\Bir}{Bir}
\DeclareMathOperator{\Cr}{Cr}
\DeclareMathOperator{\NS}{NS}
\DeclareMathOperator{\CH}{CH}
\DeclareMathOperator{\Ab}{Ab}
\DeclareMathOperator{\AJ}{AJ}
\DeclareMathOperator{\IIJ}{IJ}
\DeclareMathOperator{\JJ}{J}

\DeclareMathOperator{\Prym}{\mathbf{Prym}}
\DeclareMathOperator{\PPrym}{\mathbf{PPrym}}
\DeclareMathOperator{\Nm}{Nm}

\DeclareMathOperator{\CCH}{\mathbf{CH}}
\DeclareMathOperator{\PPic}{\mathbf{Pic}}
\DeclareMathOperator{\inj}{\hookrightarrow}
\DeclareMathOperator{\Sym}{Sym}

\DeclareMathOperator{\calF}{\mathcal{F}}
\DeclareMathOperator{\calM}{\mathcal{M}}
\DeclareMathOperator{\diag}{diag}

\DeclareMathOperator{\SL}{SL}
\DeclareMathOperator{\ch}{ch}
\DeclareMathOperator{\Sing}{Sing}

\DeclareMathOperator{\NE}{NE}

\newcommand{\Gcong}{\overset{G}{\cong}}

\newcommand\blfootnote[1]{%
  \begingroup
  \renewcommand\thefootnote{}\footnote{#1}%
  \addtocounter{footnote}{-1}%
  \endgroup
}

\begin{document}
\maketitle

\

\centerline{{\it To Professor Yuri Prokhorov on the occasion of his 60th birthday.}}

\abstract{We generalize the equivariant intermediate Jacobian torsor obstruction over $\C$, given in \cite{1}, to algebraically closed fields of characteristic zero. It is an obstruction to the (projective) linearizability problem of finite group actions on threefolds. In addition, we calculate automorphism groups of general smooth Fano threefolds of \textnumero 2.18. As an application, we prove that a general smooth Fano threefold $X$ of \textnumero 2.18 is $\Aut(X)$-linearizable.}
\section{Introduction}
Let $k$ be an algebraically closed field. We consider two problems. The first is the (projective) linearizability problem. 
\begin{dfn}(cf. \cite[Section 2]{HT22})
Let $X$ be an $n$-dimensional rational variety with a regular action of a finite group $G$. 

\begin{itemize}
\item $X$ is projectively linearizable if there exists a regular $G$-action on $\Line^n$ and a $G$-equivariant birational map
\[X\underset{G}{\overset{\sim}{\dashrightarrow}} \Line^n.\]
\item $X$ is linearizable if there exists an $n+1$-dimensional linear representation $V$ of $G$ such that $X$ and $\Line(V)$ are $G$-equivariantly birational to each other. 
\end{itemize}
\end{dfn}
 
The birational automorphism group $\Bir(\Line^n_k)$ of the $n$-dimensional projective space $\Line^n$ over $k$ is called the Cremona group of degree $n$ and is denoted by $\Cr_n(k)$.
For a finite subgroup $G$ of $\Cr_n(k)$, there exists an $n$-dimensional smooth projective rational variety $X$ with a regular and faithful $G$-action such that $\Line^n$ and $X$ are $G$-birationally equivalent, where the $G$-action on $\Line^n$ is a birational action (see \cite{dFE02}). Conversely, given an $n$-dimensional (smooth projective) rational variety with a regular and faithful $G$-action, we obtain (a conjugacy class of) a subgroup which is isomorphic to $G$ of $\Cr_n(k)$. Thus, we have a natural bijection between conjugacy classes of finite subgroups $G\subset \Cr_n(k)$ and $G$-birational equivalence classes of $n$-dimensional (smooth projective) rational varieties with regular and faithful $G$-actions. Hence, the (projective) linearizability is a property for conjugacy classes of finite subgroups $G\subset \Cr_n(k)$. I. V. Dolgachev and V. A. Iskovskikh gave the classification of conjugacy classes of finite subgroups $G\subset \Cr_2(\C)$ (\cite{DI09}). A. Pinardin, A. Sarikyan, and E. Yasinsky determined projectively linearizable conjugacy classes of finite subgroups $G\subset \Cr_2(\C)$ (\cite{PSY24}). Even the classification of involutions of $\Cr_3(\C)$ is an open problem (see \cite{Pro13}). There are some studies of (projective) linearizability of rational smooth threefolds (e.g. \cite{1} and \cite{TYZ23}). Related works on equivariant birational geometry are studies of equivariant unirationality and equivariant stable rationality of (possibly singular) Fano threefolds (e.g. \cite{CMTZ24}, \cite{CTZ25a}, \cite{CTZ25b}, and \cite{CTZ25c}).

T. Ciurca, S. Tanimoto, and Y. Tschinkel gave an obstruction of (projective) linearizability of threefolds over $\C$, called the equivariant IJT-obstruction (\cite{1}). In addition, \cite{1} gave criteria for actions of cyclic groups on conic bundles $X\to\Line^2$ with smooth quartic curve and quadric surface bundles $X\to\Line^1$, where these conic bundles are smooth Fano threefolds of \textnumero 2.18 in the Mori-Mukai classification in \cite{IP}. 

We generalize the equivariant IJT-obstruction and criteria for actions of cyclic groups on smooth Fano threefolds of \textnumero 2.18 (Section \ref{section 2}, \ref{section 3}, and \ref{section 5}), given in \cite{1}, to algebraically closed fields of characteristic zero. The equivariant IJT-obstruction is the following theorem:
\begin{thm}(Theorem \ref{IJT})\label{intro IJT}
Let $X$ be a smooth projective rational threefold over an algebraically closed field $k$ of characteristic zero with a regular and projectively linearizable action of a finite group $G$. 

Then there exists a smooth projective curve $C$ with a regular $G$-action such that, for any $G$-invariant connected component $M\subset \CCH^2_{X/k}$, we have a $G$-invariant connected component $N\subset \PPic_{C/k}$ and a $G$-equivariant isomorphism 
\[M\Gcong N. \]
\end{thm}
Here $\PPic_{C/k}$ is a fine moduli space for the relative Picard functor $\Pic_{C/k}$. The above theorem is a generalization of \cite[Theorem 1.1]{1}. The assumption for $k$ is necessary for the equivariant functorial weak factorization theorem, and this theorem is used in the proof of Theorem \ref{intro IJT}. In Section \ref{section 2}, Section \ref{section 3.2}, and Section \ref{section 3.3}, we generalize the theories of equivariant principally polarized abelian varieties and equivariant Chow group schemes, given in \cite{1}, to algebraically closed fields. Let $X$ be a rationally chain connected threefold over an algebraically closed field $k$. The Chow group scheme $\CCH^2_{X/k}$ is a separated, locally of finite type $k$-group scheme whose $k$-valued points are naturally identified with the Chow group $\CH^2(X)$ of codimension $2$ (see Section \ref{section 3.3}). We prove that a regular action of a finite group $G$ on $X$ induces a regular $G$-action on $\CCH^2_{X/k}$. In Section \ref{section 3.4}, we give a more algebraic proof of \cite[Theorem 1.1]{1}. In addition, if the equivariant functorial weak factorization theorem is proved over positive characteristics, the theorem holds over arbitrary algebraically closed fields.

As mentioned above, we have methods to study linearizability of Fano threefolds of \textnumero 2.18 with actions of (cyclic) groups. Classification of group actions on Fano threefolds of \textnumero 2.18 is a natural problem. In other words, this problem is to compute automorphism groups of smooth Fano threefolds of \textnumero 2.18, which is the second problem considered in this paper. \cite{CTT} first computed some cases in these. 
When the Picard number is $1$, automorphism groups of Fano threefolds with finite automorphism group were studied in \cite{K24}. \cite{Joe24} and \cite{PCS19} described automorphism groups of smooth Fano threefolds of \textnumero 2.21. The question of when the automorphism group of a Fano threefold is finite is solved by \cite{PCS19}. Also, \cite{PCS19} gave the classification of the identity components of automorphism groups of Fano threefolds with infinite automorphism groups.

Let $X$ be a smooth Fano threefold of \textnumero 2.18. Then we have a double cover $X\to \Line^1\times\Line^2$ branched in a smooth $(2,2)$-divisor. The variety $X$ has the following properties: 
\begin{itemize}
\item $X$ is rational.
\item $X\to \Line^1\times\Line^2\to \Line^2$ is a standard conic bundle whose discriminant curve is a plane quartic which has at worst $A_1$-singularities.
\item $X\to \Line^1\times\Line^2\to \Line^1$ is a quadric surface bundle.
\item $\Aut(X)$ is finite.
\end{itemize}
We develop formulas calculating automorphism groups of smooth Fano threefolds of \textnumero 2.18. In Section \ref{section 4.3}, we prove the following proposition:
\begin{prop}(Corollary \ref{upper of aut of 2.18})
Assume that the discriminant curve $\Delta$ is smooth. Then there exists a subgroup $G\subset \Aut(\Delta)$ such that we have an exact sequence of groups 
\[0\to  \Z/2\Z\to \Aut(X)\to  G\to 1.\]
\end{prop}
By the above proposition, for any smooth Fano threefold $X$ of \textnumero 2.18 with a smooth plane quartic curve, we obtain the inequality $|\Aut(X)|\le336$. In addition, we provide examples showing that both $|\Aut(X)|= 2\times |\Aut(\Delta)|$ and $|\Aut(X)|\neq 2\times |\Aut(\Delta)|$ can occur. In Example \ref{No.2.18 with reducible quartic}, we construct examples of smooth Fano threefolds of \textnumero 2.18 with a reducible plane quartic.

An open locus $M_{(2,2)}^{sm}$ of the GIT-quotient \[M_{(2,2)}:=|\ShO_{\Line^1\times\Line^2}(2,2)|^{ss}/\!\!/\SL_2\times\SL_3\] parametrizes the isomorphism classes of smooth Fano threefolds of \textnumero 2.18 (see Proposition \ref{prop 4.19}).  
K. DeVleming, L. Ji, P. Kennedy-Hunt, and M. H. Quek studied relations between $M_{(2,2)}$ and Fano threefolds of \textnumero 2.18 (\cite{DJK24}). As a general result, we prove that the automorphism group of a smooth Fano threefold corresponding to a general point in $M_{(2,2)}^{sm}$ is the cyclic group of order $2$: 
\begin{thm}(Theorem \ref{aut of general})
Let $X$ be a general smooth Fano threefold of \textnumero 2.18. Then, we have an isomorphism
\[\Aut(X)\cong \Z/2\Z.\]
\end{thm}
Moreover, by \cite[Example 5.14]{1}, we get the following corollary:
\begin{cor}(Corollary \ref{general linearizable})
Let $X$ be a general smooth Fano threefold of \textnumero 2.18. Then $X$ is $\Aut(X)$-linearizable. 
\end{cor}
However, there exist regular actions on smooth Fano threefolds of \textnumero 2.18 that are not projectively linearizable (see Example \ref{nolinearaction} and \cite[Remark 5.11 and Example 5.12]{1}). 

In this paper, a variety is an integral separated scheme of finite type over a field. In addition, a surface and a threefold are varieties of dimension 2 and 3, respectively. Also, a curve is a separated scheme of finite type over a field. A closed point of a scheme $X$ over an algebraically closed field is simply called a point of $X$.

\section*{Acknowledgement}
I would like to thank my supervisor, Professor Sho Tanimoto, for providing guidance. 
I am also grateful to Fumiya Okamura and Runxuan Gao for their relevant feedback. Finally, we would like to thank the referees for their detailed comments and valuable suggestions, which have significantly improved the presentation of the paper. The author was partially supported by JST FOREST program Grant number JPMJFR212Z and JST SPRING Grant number JPMJSP2125. This paper is based on the author's master's thesis \cite{Shu25}.

\section{Equivariant principally polarized abelian varieties over algebraically closed fields}\label{section 2}
We generalize the theory of $G$-equivariant principally polarized abelian varieties over $\C$ given by \cite[Section 3]{1} to arbitrary algebraically closed fields.

Let $G$ be a finite group and $k$ be an algebraically closed field. 
\begin{dfn}(\cite[Section 3]{1})
A $G$-abelian variety over $k$ is an abelian variety $A$ over $k$ with a regular $G$-action which is compatible with the group scheme structure of $A$.
If $A$ is $G$-abelian variety, then any element $g\in G$ gives an automorphism of $A$ as a homomorphism of abelian varieties.

A pair $(A,\theta_A)$ of an abelian variety $A$ and a $G$-invariant principal polarization $\theta_A\in \NS(A)^{G}$ is called a $G$-equivariant principally polarized abelian variety (or $G$-ppav).

Let $(A,\theta_A)$ and $(B,\theta_B)$ be $G$-equivariant principally polarized abelian varieties and $f:A\to B$ be a $G$-equivariant homomorphism of abelian varieties. A $G$-equivariant homomorphism $f:A\to B$ is called a $G$-equivariant homomorphism of $G$-equivariant principally polarized abelian varieties if $f$ satisfies $f^*\theta_B=\theta_A$ in $\NS(A)$.
\end{dfn}
Let $A$ be a $G$-abelian variety. Then $G$ naturally acts regularly on the dual $A^t$ of $A$.
If $L$ is a line bundle on $A$ whose class is $G$-invariant as an element in $\NS(A)$, then $L$ and $g^*L$ are algebraically equivalent for any $g\in G$. Therefore the homomorphism $\varphi_{L}:A\to A^t$, associated to the class of $L$, is $G$-equivariant. Here the homomorphism $\varphi_{L}:A\to A^t$ is a homomorphism of abelian varieties that induces the homomorphism \[\varphi_{L}(k):A(k)\to A^t(k):a\mapsto t^*_aL-L.\]
Here $t_a:A\to A$ is the translation for $a\in A(k)$.
In addition, if $f:A\to B$ is a $G$-equivariant homomorphism of $G$-varieties, then the dual of $f$
\[f^t:B^t\to A^t\]
which induces the map
\[f^t(k):B^t(k)\to A^t(k);L\mapsto f^*(L)\]
is a $G$-eqivariant homomorphism of $G$-abelian varieties.
\begin{example}
Let $C$ be a smooth projective curve over an algebraically closed field $k$ with a regular action of a finite group $G$. 
Then the Jacobian variety $\JJ(C)=\PPic^0_{C/k}$ of $C$ is a $G$-abelian variety, and 
the natural bijection $\phi^1_C:\Pic^0(C)\to \JJ(C)(k)$ is $G$-equivariant. 
Since the theta divisor of $C$ is $G$-invariant, the canonical principal polarization $\theta_C$ is $G$-invariant, i.e., the Jacobian variety $\JJ(C)$ is a $G$-equivariant principally polarized abelian variety.
\end{example}

Murre shows that a principally polarized abelian variety over an algebraically closed field $k$ is isomorphic to the product of indecomposable, principally polarized abelian varieties over $k$ \cite[Lemma 10]{4}. The following lemma and corollary are an equivariant version of \cite[Lemma 10]{4}.
\begin{lem}(\cite[Lemma 3.1]{1} and \cite[Lemma 10 $($i$)$]{4})\label{sub ppav}
Let
\[i:(A',\theta_{A'})\hookrightarrow(A,\theta_A)\]
be a $G$-equivariant immersion of $G$-equivariant principally polarized abelian varieties over an algebraically closed field. Then there exists a $G$-equivariant principally polarized abelian variety $(A'',\theta_{A''})$ such that we have a $G$-equivariant isomorphsm
\[(A,\theta_A)\Gcong (A'\times A'', p_{A'}^*\theta_{A'}+p_{A''}^*\theta_{A''})\]
as $G$-equivariant principally polarized abelian varieties, where $p_{A'}:A'\times A''\to A'$ and $ p_{A''}:A'\times A''\to A''$ are the natural projections.
\end{lem}

\begin{proof}
Let $L_A$ and $L_{A'}$ be line bundles on $A$ and $A'$ which represent $\theta_{A}$ and $\theta_{A'}$, respectively.
Since the algebraically equivalence classes $i^*\theta_A$ and $\theta_{A'}$ are equal, without loss of generality we can assume that $i^*L_A\cong L_{A'}$.
Let $f$ be the homomorphism defined as 
\[i^t\circ\varphi_{L_A}:A\to (A')^t;a\mapsto [i^*(t^*_aL_A-L_A)].\]
Then $f$ is $G$-equivariant and $f\circ i=\varphi_{L_{A'}}$.
Since $\theta_{A'}$ is principal, $f\circ i=\varphi_{L_{A'}}$ is a $G$-equivariant isomorphism. Hence $f$ is surjective and $i^{-1}(\ker (f))=0$, i.e. $A'\cap \ker (f)=0$ as $A$-schemes. 

We define $A'':=(\ker (f)_{\text{red}})^0$, then $A''$ is an abelian variety and the $G$-action on $A$ induces a regular $G$-action on $A''$ that preserves the origin because $f$ is $G$-equivariant. Therefore $A''$ is a $G$-abelian variety and $A'\cap A''=0$ holds as group schemes. Also we have $\dim A=\dim A'+\dim A''$. We denote by 
\[j:A''\inj A\]
the natural $G$-equivariant immersion, then 
\[\rho:A'\times A''\to A;(a',a'')\mapsto i(a')+j(a''),\]
is a $G$-equivariant isomorphism of $G$-abelian varieties. Indeed, since $i$ and $j$ are injective and $A'\cap A''=0$ holds, $\rho$ is injective. From this and $\dim A=\dim A'+\dim A''$, we see that the morphism $\rho$ is surjective. The $G$-eqivariant homomorphism $\rho$ identifies $A$ with $A'\times A''$ as $G$-varieties.

Let $a''\in A''$ be a closed point and $D=L_A-p^*_{A'}L_{A'}$. Then the following diagram is commutative:
\[
  \begin{diagram}
    \node{A'\times\{a"\}} \arrow{e,t}{t_{(0,-a")}} \arrow{s,l}{\rho} \node{A'} \arrow{s,r}{i} \\
    \node{A}\arrow{ne,t}{p_{A'}}  \node{A}\arrow{w,r}{t_{(0,a")}}
  \end{diagram}
\]
Because we have $a''\in A''(k)=(\ker (i^t\circ \varphi_{L_A}))^0(k)$, we get
\[i^*t^*_{(0,a'')}L_{A}- L_{A'}=i^t\circ \varphi_{L_A}\circ j(a'')=0\in (A')^t.\]
Hence we have
\[i^*t^*_{(0,a'')}L_{A}\cong L_{A'}.\]
From the above isomorphism and the fact that
\[\rho =t_{(0,a'')}\circ i\circ t_{(0,-a'')}\]
holds on $A'\times\{a''\}$, we obtain
\begin{equation}\label{11}
(L_A)_{a''}=(t_{(0,-a'')}|_{A'\times\{a''\}})^*i^*t^*_{(0,a'')}L_{A}\cong (t_{(0,-a'')}|_{A'\times\{a''\}})^*L_{A'}.
\end{equation}
Here $(L_A)_{a''}$ is the restriction of $L_A$ to $A'\times \{a''\}$. In addition, $t_{(0,-a'')}=p_{A'}\circ \rho$ on $A'\times\{a''\}$, hence we have
\begin{equation}\label{12}
(p_{A'}^*L_{A'})_{a''}=(t_{(0,-a'')}|_{A'\times\{a''\}})^*L_{A'}.
\end{equation}
From the equations (\ref{11}) and (\ref{12}), we find that
\[D_{a''}=(L_A)_{a''}-(p_{A'}^*L_{A'})_{a''}\]
is trivial.
Let $L_{A''}$ be the $\ShO_{A''}$-module $(p_{A''})_*D$. Then $L_{A''}$ is a line bundle on $A''$ and we have  
\[L_A\cong p_{A'}^*L_{A'}+ p_{A''}^*L_{A''}\]
cf. \cite[III Exercise 12.4]{2}. Let $\theta_{A''}$ be the class of $L_{A''}$ in $\NS(A'')$. Since the class of $D$ is $G$-invariant as an element in $\NS(A)$ and $p_{A''}$ is $G$-equivariant, the class $\theta_{A''}$ is $G$-invariant as an element in $\NS(A'')$.
Set $n=\dim A,n'=\dim A',n''=\dim A''$. Since $\theta_A$ and $\theta_{A'}$ are principal on $A$ and $A'$, respectively, by the Riemann-Roch theorem for the line bundles $L_A$ and $L_{A'}$ \cite[Section 16, p.140]{9}, we get
\[n!=L_A^{n}=\binom{n}{n'}L_{A'}^{n'}\cdot L_{A''}^{n''}=\frac{n!}{n''!}L_{A''}^{n''}.\]
Therefore we obtain $L_{A''}^{n''}=n''!$, i.e. $\theta_{A''}$ is principal on $A''$ by the Riemann-Roch theorem for the line bundle $L_{A''}$.
\end{proof}
An indecomposable $G$-equivariant principally polarized abelian variety is a non-zero $G$-equivariant principally polarized abelian variety that is not isomorphic to the product of two non-zero $G$-equivariant principally polarized abelian varieties. By Lemma \ref{sub ppav}, we have the following corollary: 
\begin{cor}(\cite[Corollary 3.2]{1} and \cite[Lemma 10 $($ii$)$]{4})\label{decomposition of ppav}
Any $G$-equivariant principally polarized abelian variety over an algebraically closed field has a unique decomposition into the product of indecomposable $G$-equivariant principally polarized abelian varieties up to permutation of the factors. 
\end{cor}
An indecomposable $G$-equivariant principally polarized abelian variety is also called irreducible.
\begin{example}(\cite{1})
Let $G$ be a finite group and $C$ be a smooth projective $G$-irreducible $G$-curve of genus $g\ge 1$. Then, the $G$-equivariant Jacobian $\JJ(C)$ is an irreducible $G$-equivariant principally polarized abelian variety. 

Here a $G$-curve is $G$-irreducible if the $G$-action on irreducible components of the curve is transitive \cite[in proof of Theorem 3.3]{1}.
\end{example}

\section{Equivariant intermediate Jacobian torsor obstructions}\label{section 3}
In this section, assume that $k$ is an algebraically closed field. We introduce an equivariant version of Murre's intermediate Jacobian over $k$, called an equivariant Murre's intermediate Jacobian. In addition, we show that a principal polarization on an equivariant Murre's intermediate Jacobian given by \cite{3} is a $G$-invariant principal polarization. 

T. Ciurca, S. Tanimoto, and Y. Tschinkel constructed an obstruction to the projective linearizability of a finite group action over $\C$, called the equivariant IJT-obstruction \cite{1}.
We generalize the equivariant IJT-obstruction to an algebraically closed field $k$ of characteristic zero. See \cite[(11.5)]{GW10} for the definition and properties of torsors.

\subsection{Murre's intermediate Jacobians}\label{Murre IJ}
Following \cite[Section 2]{3}, we introduce Murre's intermediate Jacobian over $k$.
A smooth projective (rationally chain connected) variety $X$ over $k$ is associated with an abelian variety $\Ab^2(X)$ over $k$, called the algebraic representative for algebraically trivial codimension $2$ cycles on $X$. Over an algebraically closed field $k$, the abelian variety $\Ab^2(X)$ is constructed by Murre \cite{14}. Let $\CH^i(X)$ be the Chow group of codimension $i$ on a smooth projective variety $X$ over $k$. We define $\CH^i(X)_{alg}\subset \CH^i(X)$ to be the subgroup of cycles algebraically equivalent to $0$. 
\begin{dfn}(\cite[Section 2]{BS83})
Let $X$, $A$ be a smooth projective variety and an abelian variety over $k$, respectively. A homomorphism 
\[\phi: \CH^i(X)_{alg}\to A(k)\]
is regular if for any smooth connected variety $T$ over $k$, $t_0\in T(k)$ and codimension-$i$ cycle $Z\in Z^i(T\times X)$, the composition 
\begin{align*}
T(k)&\to \CH^i(X)_{alg}\to A(k)\\
t&\mapsto Z_t-Z_{t_0}\mapsto \phi(Z_t-Z_{t_0})
\end{align*}
is induced by a morphism $T\to A$.
\end{dfn}
By the following theorem, the abelian variety $\Ab^2(X)$ associated to $X$ is unique up to isomorphism.
\begin{thm}(\cite[Proposition 2.3]{3})\label{AJ}
Let $X$ be a smooth projective variety over $k$. Then there exists a group morphism 
\[\phi_X^2:\CH^2(X)_{alg}\to \Ab^2(X)(k)\]
that is surjective and initial among regular homomorphisms with values in an abelian variety over $k$. Moreover, if $X$ is rationally chain connected, then $\phi_X^2$ is bijective.
\end{thm}
For a morphism $g:Y\to X$ between smooth projective varieties over $k$, by the universal property of regular homomorphisms, there exists a unique morphism \begin{equation}\label{g+}
g^+:\Ab^2(X) \to\Ab^2(Y)\end{equation} of abelian varieties over $k$ such that $g^+(k)\circ \phi_X^2=\phi_Y^2\circ g^*$ holds.
If $X$ and $Y$ are smooth projective varieties over $k$ of the same dimension and $f:Y\to X$ is a morphism of varieties over $k$, then there exists a unique morphism \begin{equation}\label{f+}f_+:\Ab^2(Y) \to\Ab^2(X)\end{equation} of abelian varieties over $k$ such that $\phi_X^2\circ f_*=f_+(k)\circ \phi_Y^2$ holds.

Benoist-Wittenberg constructed a principal polarization $\theta_X$ on the abelian variety $\Ab^2(X)$ associated to a smooth projective rationally chain connected threefold $X$ over $k$ \cite[Section 2.3]{3}. 

We assume that $X$ is a smooth projective rationally chain connected threefold.
\begin{property}\cite[Property 2.4]{3}
There exists $\theta\in \NS(\Ab^2 (X))$ satisfying the following assertions.
\begin{itemize}
\item[$($i$)$] For all prime numbers $\ell$ invertible in $k$, the image $c_{1,\ell}(-\theta)$ of $-\theta$ by the $\ell$-adic first Chern class
\[c_{1,\ell}:\NS(\Ab^2(X))\hookrightarrow H^2_{\text{\'{e}t}}(\Ab^2(X),\Z_{\ell}(1))=\left(\bigwedge^2H^1_{\text{\'{e}t}}(\Ab^2(X),\Z_{\ell})\right)(1)\]
corresponds, via the isomorphism 
\[T_{\ell}(\lambda^2\circ(\phi_X^2)^{-1}):H^1_{\text{\'{e}t}}(\Ab^2(X),\Z_{\ell})^{\vee}\to 
H^3_{\text{\'{e}t}}(X,\Z_{\ell}(2))/(\text{torsion}), \]
to the cup product map
\[\bigwedge^2H^3_{\text{\'{e}t}}(X,\Z_{\ell}(2))\to H^6_{\text{\'{e}t}}(X,\Z_{\ell}(4))\overset{\text{deg}}{\longrightarrow}\Z_{\ell}(1).\]
\item[$($ii$)$]
The class $\theta\in \NS(\Ab^2 (X))$ is a principal polarization of $\Ab^2(X)$.
\end{itemize}
\end{property}
Here $\lambda^2$ is Bloch's Abel-Jacobi map and $T_{\ell}$ is the $\ell$-adic Tate functor (see \cite[Section 2.2 and Section 2.3]{3}).
Property A only depends on $X$ and the class $\theta$ in Property A $($i$)$ is unique since $c_{1,\ell}$ is injective \cite[Section 2.3]{3}.
If $X$ satisfies Property A, then $(\Ab^2(X),\theta)$ is a principally polarized abelian variety over $k$, which we denote by $\IIJ(X)$ and call the intermediate Jacobian of $X$.

\begin{thm}(\cite[Corollary 2.8 and Proposition 2.5]{3})\label{BW20, prop2.5}
\begin{itemize}
\item[(1)] A smooth projective rational threefold over $k$ satisfies Property A.
\item[(2)] When $\ch k=0$, a smooth projective rationally connected threefold over $k$ satisfies Property A.
\end{itemize}
\end{thm}

\subsection{Equivariant Murre's intermediate Jacobians}\label{section 3.2}
In this section, we introduce an equivariant abelian variety $\Ab^2(X)$ over $k$ associated to a smooth projective rationally chain connected threefold $X$ over $k$. Moreover, we show that the morphisms introduced in Section \ref{Murre IJ} are equivariant.

Let $X$ be a smooth projective threefold over $k$ with a regular action of a finite group $G$. Then the $G$-action on $X$ induces a $G$-action on $\CH^2(X)$ and $\CH^2(X)_{alg}$.
\begin{lem}\label{action on Ab}
The $G$-action on $\CH^2(X)_{alg}$ induces a regular $G$-action on $\Ab^2(X)$ by the universal property of regular homomorphisms, and $\Ab^2(X)$ is a $G$-abelian variety for this $G$-action on $\Ab^2(X)$. Moreover, $\phi_X^2$ is $G$-equivariant. 
\end{lem}
\begin{proof}
For each $g\in G$, since $g$ induces an automorphism of $X$ over $k$, by the universal property of regular homomorphimsms (Theorem \ref{AJ}), the morphism $g_+$ is an automorphism of $\Ab^2(X)$ as abelian varieties over $k$. Hence $G$ regularly acts on $\Ab^2(X)$ by the morphism $G\times \Ab^2(X)\to\Ab^2(X)$ over $k$ which induces the map
\[G\times \Ab^2(X)(k)\to \Ab^2(X)(k);(g,x)\mapsto g_+(x).\] 
Then the group morphism 
\[\phi^2_X:\CH^2(X)_{alg}\to \Ab^2(X)(k)\]
is $G$-equivariant and the $G$-action preserves the origin, i.e., $\Ab^2(X)$ is a $G$-abelian variety.
\end{proof}
\begin{cor}\label{f^+ is equivariant}
Let $f:Y\to X$ be a $G$-equivariant morphism of smooth projective rationally chain connected threefolds over an algebraically closed field $k$ with a regular action of a finite group $G$. Then $f_+$ and $f^+$ are $G$-equivariant.
\end{cor}
\begin{proof}
By Lemma \ref{action on Ab}, $\phi^2_X$ and $\phi^2_Y$ are $G$-equivariant. Because $f$ is $G$-equivariant, so are $f^*$ and $f_*$. Since these facts hold and we have the equations 
\[f_+(k)=\phi^2_X\circ f_*\circ(\phi^2_Y)^{-1}\]
and
\[f^+(k)=\phi^2_Y\circ f^*\circ(\phi^2_X)^{-1},\]
the morphisms $f_+$ and $f^+$ are $G$-equivariant.
\end{proof}

Let $X$ be a smooth projective rationally chain connected threefold with a regular $G$-action and satisfies Property A. Then $G$ acts on \'etale cohomology, $\CH^2(X)$ and $\CH^2(X)\{\ell\}$  for all prime numbers $\ell$ invertible in $k$. Here $\CH^2(X)\{\ell\}\subset \CH^2(X)$ is the $\ell$-torsion subgroup.
By Lemma \ref{action on Ab}, the finite group $G$ regularly acts on $\Ab^2(X)$. Hence
$G$ acts naturally on $\NS(\Ab^2(X))$ and \'etale cohomology. By Lemma \ref{action on Ab} and \cite[Proposition 3.3]{B79}, the morphisms $\phi_X^2$ and $\lambda^2$ are $G$-equivariant, hence so is $T_{\ell}(\lambda^2\circ(\phi_X^2)^{-1})$ for all prime numbers $\ell$. By construction, for any prime number $\ell$, the $\ell$-adic first Chern class and the cup product map are $G$-equivariant. Therefore the principal polarization $\theta$ is $G$-invariant, i.e., the intermediate Jacobian $\IIJ(X)$ is a $G$-equivariant principally polarized abelian variety over $k$.

By Theorem \ref{BW20, prop2.5} and Lemma \ref{action on Ab}, we obtain the following corollary:

\begin{cor}
Let $X$ be a smooth projective rationally chain connected threefold with a regular $G$-action over $k$. In either of the following cases, the abelian variety $\Ab^2(X)$ has the $G$-equivariant principally polarization $\theta_X$. 
\begin{itemize}
\item[(1)] $\ch k=0$
\item[(2)] $X$ is rational.
\end{itemize}
\end{cor}

\begin{dfn}
Let $X$ be a smooth projective rationally chain connected threefold with a regular $G$-action over $k$. Assume that $X$ satisfies Property A. Then we call $\IIJ(X)=(\Ab^2(X), \theta_X)$ the $G$-equivariant intermediate Jacobian.
\end{dfn}

\subsection{Chow group schemes}\label{section 3.3}
\begin{dfn}(\cite[Section 2]{1})
Let $X$ be a smooth projective rationally chain connected threefold over $k$. For each $\gamma\in \NS^2(X)=\CH^2(X)/\CH^2(X)_{alg}$, represented by a $1$-cycle $Z_0$, the bijections 
\[(\CH^2(X))^{\gamma}:=[Z_0]+\CH^2(X)_{alg}\overset{\text{transitive}}{\cong}\CH^2(X)_{alg}\overset{\phi^2_X}{\cong}\Ab^2(X)(k)\]
define a scheme structure on $(\CH^2(X))^{\gamma}$, denoted by $(\CCH^2_{X/k})^{\gamma}$. The first bijection is unique up to unique translation by an element in $\CH^2(X)_{alg}$. Then these induce a group scheme structure on $\CH^2(X)$, and we denote this scheme by $\CCH^2_{X/k}$. 
\end{dfn}
The following lemma asserts that a regular $G$-action on a smooth projective rationally connected threefold $X$ induces a regular $G$-action on $\CCH^2_{X/k}$ and $(\CCH^2_{X/k})^0$.

\begin{lem}(\cite[Lemma 2.1 and Lemma 2.2]{1})\label{action on CH}
Let $X$ be a smooth projective rationally chain connected threefold over $k$ with a regular action of a finite group $G$.
\begin{itemize}
\item[$(1)$] The $G$-action on $\CH^2(X)_{alg}$ induces a regular $G$-action on $(\CCH^2_{X/k})^{0}$ by the universal property of regular homomorphisms. Moreover, the $G$-variety $(\CCH^2_{X/k})^{0}$ is a $G$-abelian variety. 
\item[$(2)$] The $G$-action on $\CH^2(X)$ induces a regular $G$-action on $\CCH^2_{X/k}$.
\end{itemize}
\end{lem}
\begin{proof}
\begin{itemize}
\item[$(1)$]
Let us recall that the regular $G$-action on $\Ab^2(X)$ is given by the map
\[G\times \Ab^2(X)(k)\to \Ab^2(X)(k);(g,x)\mapsto g_+(k)(x),\] 
and, for any $g \in G$, the following diagram is commutative:
\[
  \begin{CD}
     \CH^2(X)_{alg} @>{g_*}>> \CH^2(X)_{alg} \\
  @V{\phi_X^2}VV    @V{\phi_X^2}VV \\
     \Ab^2(X)(k)   @>{g_+(k)}>>  \Ab^2(X)(k).
  \end{CD}
\]
By the facts above and Lemma \ref{action on Ab}, the map
\[G\times (\CCH^2_{X/k})^{0}(k)\to (\CCH^2_{X/k})^{0}(k);(g,Z)\mapsto gZ\] 
yields a regular $G$-action on $(\CCH^2_{X/k})^{0}$ and the $G$-action preserves the origin, i.e. $(\CCH^2_{X/k})^{0}$ is a $G$-abelian variety.
\item[$(2)$]
For each $g\in G$ and $\gamma \in \NS^2(X)$, let $g$ be the map 
\[(\CCH^2_{X/k})^{\gamma}(k)\to(\CCH^2_{X/k})^{g\gamma}(k);[Z]\to [gZ].\]
Then the following diagram is commutative:
\[
  \begin{CD}
     (\CCH^2_{X/k})^{\gamma}(k) @>{g}>> (\CCH^2_{X/k})^{g\gamma}(k) \\
  @V{\text{transitive}}VV    @V{\text{transitive}}VV \\
     (\CCH^2_{X/k})^0(k)   @>{g(k)}>>  (\CCH^2_{X/k})^0(k).
  \end{CD}
\]
Hence, the map $g:(\CCH^2_{X/k})^{\gamma}(k)\to(\CCH^2_{X/k})^{g\gamma}(k)$ is algebraic and we get a morphism
$g:(\CCH^2_{X/k})^{\gamma}\to(\CCH^2_{X/k})^{g\gamma}$ which induces the map
\[g(k):(\CCH^2_{X/k})^{\gamma}(k)\to(\CCH^2_{X/k})^{g\gamma}(k);[Z]\to [gZ].\]
Glueing them, we get a morphism 
\[\rho:G\times \CCH^2_{X/k}\to \CCH^2_{X/k}.\]
The restriction of $\rho$ to $\{g\}\times (\CCH^2_{X/k})^{\gamma}$ is the morphism $\{g\}\times (\CCH^2_{X/k})^{\gamma}\to(\CCH^2_{X/k})^{g\gamma}$ which induces the map
\[g|_{\{g\}\times (\CCH^2_{X/k})^{\gamma}}(k):(\CCH^2_{X/k})^{\gamma}(k)\to(\CCH^2_{X/k})^{g\gamma}(k);[Z]\to [gZ].\]
Therefore $\rho$ is a $G$-action on $\CCH^2_{X/k}$ and the $G$-action $\rho$ is regular.
\end{itemize}
\end{proof}
If $X$ satisfies Property A, the $G$-abelian variety $(\CCH^2_{X/k})^{0}$ has a $G$-invariant principal polarization induced by the $G$-equivariant intermediate Jacobian $\IIJ(X)$, i.e., $(\CCH^2_{X/k})^{0}$ is a $G$-equivariant principally polarized abelian variety over $k$.
We identify $(\CCH^2_{X/k})^{0}$ with $\IIJ(X)$ as $G$-equivariant principally polarized abelian varieties over $k$.

\subsection{Obstructions of linearizability}\label{section 3.4}
In this section, assume that $k$ is an algebraically closed field of characteristic zero. By the assumption, we see that the intermediate Jacobian of rationally connected threefold is defined. This fact is necessary to prove Lemma \ref{BW20 lemma 2.9} and \ref{before blowup}. Additionally, the equivariant functorial weak factorization theorem holds over an algebraically closed field of characteristic zero(\cite[Theorem 5-2-1]{6} or \cite[Theorem  1.3.3 (1)]{5}), and this theorem implies Theorem \ref{important lemma}.
The relative Picard functor of a smooth projective curve $C$ over $k$ is represented by a separated, locally of finite type scheme over $k$. This scheme is denoted by $\PPic_{C/k}$. Let $\NS(C)$ be the N\'eron-Severi group of $C$. For a class $n\in \NS(C)$, the connected component corresponding to $n$ is represented by $\PPic^n_{C/k}$. If a finite group $G$ acts on $C$, the group $G$ acts on $\PPic_{C/k}$, $\PPic^0_{C/k}$, and $\NS(C)$. 
\begin{dfn}(\cite[Definition 3.4]{1})
Let $X$ be a smooth projective rationally connected threefold over $k$ with a regular action of a finite group $G$. Then the $G$-equivariant IJT-obstruction vanishes for $X$ if there exists a smooth projective curve $C$ with a regular $G$-action such that, for any $G$-invariant class $\gamma\in \NS^2(X)$, we have a $G$-equivariant isomorphism
\[(\CCH^2_{X/k})^{\gamma}\Gcong \PPic^n_{C/k}\]
for some $G$-invariant class $n\in \NS(C)$. Moreover, when $\gamma=0$, we have a $G$-equivariant isomorphism
\[(\CCH^2_{X/k})^{0}\Gcong \PPic^0_{C/k}\]
as $G$-equivariant principally polarized abelian varieties.
\end{dfn}
The following theorem asserts that the IJT-obstruction is an obstruction to projective linearizability.
\begin{thm}(\cite[Theorem 3.3]{1})\label{IJT}
Let $X$ be a smooth projective rational threefold over $k$ with a regular and projectively linearizable action of a finite group $G$. 
Then there exists a smooth projective curve $C$ with a regular $G$-action such that for any $G$-invariant class $\gamma\in \NS^2(X)$, we have a $G$-equivariant isomorphism 
\[(\CCH^2_{X/k})^{\gamma}\Gcong \PPic^n_{C/k}\]
for some $G$-invariant class $n$. When $\gamma=0$, we get a $G$-equivariant isomorphism
\[(\CCH^2_{X/k})^{0}\Gcong \PPic^0_{C/k}\]
as $G$-equivariant principally polarized abelian varieties.

In particular, the $G$-equivariant IJT-obstruction vanishes for $X$.
\end{thm}

To prove Theorem \ref{IJT}, we give a relation between the equivariant intermediate Jacobian of a smooth projective rational threefold with a regular action of a finite group and the equivariant intermediate Jacobian of an equivariant blow-up or blow-down of it.

\begin{lem}(\cite[Lemma 2.9]{3})\label{BW20 lemma 2.9}
Let $f:Y\to X$ be a $G$-equivariant birational morphism of smooth projective rationally chain connected threefolds over $k$ with a regular action of a finite group $G$. Then the morphism, given by \eqref{f+}, 
\[f^+:\IIJ(X)\inj\IIJ(Y)\]
is a $G$-equivariant immersion of $G$-equivariant principally polarized abelian varieties.
\end{lem}

\begin{proof}
For any smooth projective rationally chain connected variety $V$ over $k$, the group $\CH_0(V)_{\Q}$ is supported in dimension 1. 
By \cite[Lemma 2.9]{3} and Lemma \ref{f^+ is equivariant}, the morphism $f^+$ is a $G$-equivariant immersion of $G$-equivariant principally polarized abelian varieties.
\end{proof}

\begin{lem}(\cite[Lemma 2.10]{3})\label{before blowup}
Let $X$ be a smooth projective rationally chain connected threefold over $k$ with a regular action of a finite group $G$ and $f:Y\to X$ be the $G$-equivariant blow-up of a smooth closed subscheme $C\subset X$ which is $G$-irreducible and has dimension $1$. 

Then there exist a $G$-equivariant injection
\[z^*:\Pic(C)\inj \CH^2(Y)\]
and a $G$-equivariant immersion 
\[z^+:\JJ(C)\inj \IIJ(Y)\]
of $G$-equivariant principally polarized abelian varieties such that these satisfy the following properties.
\begin{itemize}
\item[$(1)$]
The group homomorphism
\[(f^*,z^*):\CH^2(X)\times\Pic(C) \to \CH^2(Y)\]
is a $G$-equivariant isomorphism and the restriction of $(f^*,z^*)$ to algebraically trivial cycles
\[(f^*,z^*):\CH^2(X)_{alg}\times\Pic^0(C) \to \CH^2(Y)_{alg}\]
is a $G$-equivariant isomorphism.
\item[$(2)$]
The morphism
\[(f^+,z^+):\IIJ(X)\times \JJ(C)\to \IIJ(Y)\]
is a $G$-equivariant isomorphism of $G$-equivariant principally polarized abelian varieties.
\item[$(3)$] The following diagram is commutative:
\[
  \begin{CD}
     \CH^2(X)_{alg}\times\Pic^0(C) @>{(f^*,z^*)}>> \CH^2(Y)_{alg} \\
  @V{\phi_X^2\times \phi_C^1}VV    @V{\phi_Y^2}VV \\
     \IIJ(X)(k)\times \JJ(C)(k)   @>{(f^+,z^+)(k)}>>  \IIJ(Y)(k).
  \end{CD}
\]
\end{itemize}
\end{lem}

\begin{proof}
Let $E$ be the $G$-invariant exceptional divisor of the blow-up $f$ and $i:E\inj Y$ be the $G$-equivariant natural immersion, and the $G$-invariant cycle $z\in \CH^2(Y\times C)$ be $(i, f|_E)_*E$. We define $z^*:\Pic(C)\to\CH^2(Y)$ to be the correspondence defined by $z$. 
Because $C$ is a smooth center of the blow-up $f$ of the smooth projective threefold $X$, $C$ is a smooth closed codimension-2 subscheme $X$. 
Then $(f^*, z^*):\CH^2(X)\times\Pic(C) \to \CH^2(Y)$ and the restriction of $(f^*,z^*)$ to algebraically trivial cycles $(f^*,z^*):\CH^2(X)_{alg}\times\Pic^0(C) \to  \CH^2(Y)_{alg}$ are isomorphisms. By Theorem \ref{AJ} for $Y$, we have a homomorphism $z^+:\JJ(C)\to \Ab^2(Y)$ of abelian varieties such that $\phi^2_Y\circ z^*|_{\Pic^0(C)}=z^+\circ\phi^1_{C}$, where $\phi^1_{C}:\Pic^0(C)\to \JJ(C)(k)$ is the natural morphism (See \cite[Lemma 2.10]{3}). Then, by \cite[Lemma 2.10]{3}, the homomorphism $(f^+,z^+):\IIJ(X)\times \JJ(C)\to \IIJ(Y)$ is an isomorphism of principally polarized abelian varieties. Furthermore the following diagram is commutative:
\[
  \begin{CD}
     \CH^2(X)_{alg}\times\Pic^0(C) @>{(f^*,z^*)}>> \CH^2(Y)_{alg} \\
  @V{\phi_X^2\times \phi_C^1}VV    @V{\phi_Y^2}VV \\
     \IIJ(X)(k)\times \JJ(C)(k)  @>{(f^+,z^+)(k)}>>  \IIJ(Y)(k).
  \end{CD}
\]
Here, the principal polarization $\theta_Y$ of $\IIJ(Y)$ is given by $\theta_Y=(f^+,z^+)_*(\theta_X,\theta_C)$. See \cite[Lemma 2.10]{3} for details and notations. 
Since $E$ is $G$-invariant and the morphisms $i$ and $f$ are $G$-equivariant, $z$ is $G$-invariant. Hence $z^*$ is $G$-equivariant. Because $\phi_Y^2,\phi_C^1$ and $z^*$ are $G$-equivariant, so is $z^+$. Since the morphisms $(f^*,z^*)$, $\phi_X^2\times \phi_C^1$ and $\phi_Y^2$ are $G$-equivariant, so is $(f^+,z^+)$. Therefore $(f^+,z^+)$ is a $G$-equivariant isomorphism of $G$-equivariant principally polarized abelian varieties.
\end{proof}

\begin{thm}\label{important lemma}(\cite[Theorem 1.1 and Theorem 3.3]{1})
Let $X$ be smooth projective rationally connected threefold with a regular action of a finite group $G$. 
Assume that a smooth projective rationally connected $3$-dimensional $G$-variety $Y$ is $G$-equivariantly birational to $X$. Then there exist smooth projective $G$-curves $C_1$ and $C_2$ such that, for any $G$-invariant elements $n\in \NS(C_1)$ and $\gamma_1\in \NS^2(Y)$, we have a $G$-equivariant morphism 
\[(\CCH^2_{Y/k})^{\gamma_1}\times \PPic^n_{C_1/k}\cong (\CCH^2_{X/k})^{\gamma_2}\times\PPic^m_{C_2/k}\]
of $G$-varieties for some $G$-invariant classes $m\in \NS(C_2)$ and $\gamma_2\in \NS^2(X)$, and a $G$-equivariant morphism 
\[(\CCH^2_{Y/k})^{0}\times \PPic^0_{C_1/k}\cong (\CCH^2_{X/k})^{0}\times\PPic^0_{C_2/k}\]
of $G$-equivariant principally polarized abelian varieties. In addition, the following diagram is commutative:
\[
  \begin{CD}
     (\CCH^2_{X/k})^{\gamma_1}\times \PPic^n_{C_1/k} @>{\overset{G}{\cong}}>> (\CCH^2_{Y/k})^{\gamma_2}\times\PPic^m_{C_2/k} \\
  @V{\text{transitive}}VV    @V{\text{transitive}}VV \\
    (\CCH^2_{X/k})^{0}\times \PPic^0_{C_1/k}   @>{\overset{G}{\cong}}>>  (\CCH^2_{Y/k})^{0}\times\PPic^0_{C_2/k}.
  \end{CD}
\] 
Here the vertical morphisms are transitive, i.e. these preserve direct products. 
\end{thm}

\begin{proof}
According to the assumption, we have a $G$-equivariant birational map \[\varphi:Y\underset{G}{\overset{\sim}{\dashrightarrow}} X.\]
By the equivariant functorial weak factorization theorem(\cite[Theorem 5-2-1]{6} or \cite[Theorem  1.3.3 (1)]{5}), there exists a sequence of $G$-equivariant 
birational maps between smooth projective threefolds with regular $G$-actions
\[Y=Y_n\overset{\psi_{n}}{\dashrightarrow}Y_{n-1}\overset{\psi_{n-1}}{\dashrightarrow}\cdots \overset{\psi_{2}}{\dashrightarrow} Y_1\overset{\psi_{1}}{\dashrightarrow}Y_{0}=X\]
such that $\varphi=\psi_n\circ\cdots\circ\psi_{1}$ and $\psi_{i}$ is a $G$-equivariant blow-up or down with a $G$-irreducible smooth center of $Y_{i-1}$ for $i=1,...,n$. We say that $Y_i$ satisfies the statement if, when $Y=Y_i$, the assertion holds. Since $X$ satisfies the statement, by the induction on $n$, the assertion is reduced to a single $G$-equivariant blow-up
\[\psi:Y_2\to Y_1\]
with a $G$-irreducible smooth center $\Gamma$ where one of the $Y_i$'s satisfies the statement, and we need to prove the statement for the other $Y_j$.\\

Since the smooth center $\Gamma$ of $\psi$ is $G$-irreducible, the dimension and the genus of connected components of $\Gamma$ are constant.
If $\Gamma$ has dimension $0$ or connected components of $\Gamma$ have genus $0$, then we have the blow-up formula
\begin{equation}\label{blow up formula 0}
\CH^2(Y_1)\oplus M\overset{G}{\cong} \CH^2(Y_2)
\end{equation}
which is $G$-equivariant isomorphism for some finitely generated $\Z$-module $M$ with a $G$-action. Restricting to algebraically trivial cycles, we get the $G$-equivariant isomorphism  
\[\psi^*:\CH^2(Y_1)_{alg}\to \CH^2(Y_2)_{alg}.\]
Here we identify $\CH^2(Y_1)_{alg}$ and $\CH^2(Y_1)_{alg}\times\{\overset{\rightarrow}{n}\}$ for some $\overset{\rightarrow}{n}\in M$ (where $\overset{\rightarrow}{n}$ parametrizes algebraocally trivial class in $M=\CH^1(\Gamma)$).
By the equation $\phi_{Y_2}^2\circ \psi^*=\psi^+(k)\circ \phi_{Y_1}^2$ and Lemma \ref{BW20 lemma 2.9}, we obtain the $G$-equivariant isomorphism
\[\psi^+:\IIJ^2(Y_1)\to\IIJ^2(Y_2)\]
of $G$-equivariant principally polarized abelian varieties. According to the isomorphism (\ref{blow up formula 0}), for any $G$-invariant class $\gamma_2\in \NS^2(Y_2)$, there exists a $G$-invariant class $\gamma_1\in \NS^2(Y_1)$ and we get the following commutative diagram:
\[
  \begin{CD}
     (\CH^2(Y_1))^{\gamma_1} @>{\overset{G}{\cong}}>> (\CH^2(Y_2))^{\gamma_2} \\
  @V{\text{transitive}}VV    @V{\text{transitive}}VV \\
     (\CH^2(Y_1))^{0}   @>{\psi^*}>>  (\CH^2(Y_2))^{0}\\
  @V{\phi_{Y_1}^2}VV    @V{\phi_{Y_2}^2}VV \\
     \Ab^2(Y_1)(k)   @>{\psi^+(k)}>>  \Ab^2(Y_2)(k).
  \end{CD}
\]
Hence the following diagram is commutative:
\begin{equation}\label{abc}
  \begin{CD}
     (\CCH^2_{Y_2/k})^{\gamma_2} @>{\overset{G}{\cong}}>> (\CCH^2_{Y_1/k})^{\gamma_1} \\
  @V{\text{transitive}}VV    @V{\text{transitive}}VV \\
     (\CCH^2_{Y_2/k})^{0}   @>{\overset{G}{\cong}}>>  (\CCH^2_{Y_1/k})^{0}.
  \end{CD}
\end{equation}
Here, the upper arrow is a $G$-equivariant morphism, the lower arrow is a $G$-equivariant morphism of $G$-equivariant principally polarized abelian varieties and the vertical arrows are transitive. Similarly, for any $G$-invariant class $\gamma_1\in \NS^2(Y_1)$, there exists a $G$-invariant class $\gamma_2\in \NS^2(Y_2)$ such that we have the same commutative diagram as in (\ref{abc}). Therefore if one of $Y_i$'s satisfies the statement, then the other $Y_j$ does as well.

Assume that $\Gamma$ has dimension $\dim(\Gamma)=1$ and each connected component of $\Gamma$ has genus $g\ge1$. 
If $Y_1$ satisfies the statement, then there exist smooth projective $G$-curves $C, D_1$ such that we have the following property.
\begin{itemize}
\item[(i)]
 For any $G$-invariant classes $\gamma_1\in \NS^2(Y_1), m_1\in \NS(D_1)$, we have $G$-invariant classes $\delta\in \NS^2(X), n\in\NS(C)$, $G$-isomorphisms
\[(\CCH^2_{Y_1/k})^{\gamma_1}\times \PPic^{m_1}_{D_1/k}\cong (\CCH^2_{X/k})^{\delta}\times\PPic^n_{C/k},\]
\[(\CCH^2_{Y_1/k})^{0}\times \PPic^{0}_{D_1/k}\cong (\CCH^2_{X/k})^{0}\times\PPic^0_{C/k}\]
and the following commutative diagram:
\begin{equation*}
  \begin{CD}
     (\CCH^2_{Y_1/k})^{\gamma_1}\times \PPic^{m_1}_{D_1/k} 
     @>{\overset{G}{\cong}}>>
     (\CCH^2_{X/k})^{\delta}\times\PPic^n_{C/k} \\
  @V{\text{transitive}}VV    @V{\text{transitive}}VV \\
     (\CCH^2_{Y_1/k})^{0}\times \PPic^{0}_{D_1/k}   @>{\overset{G}{\cong}}>>  (\CCH^2_{X/k})^{0}\times\PPic^0_{C/k}.
  \end{CD}
\end{equation*}
Here, the vertical morphisms are transitive, and the lower isomorphism is an isomorphism as $G$-equivariant principally polarized abelian varieties.
\end{itemize}

By Lemma \ref{before blowup}, for any $G$-invariant class $\gamma_2\in \NS^2(Y_2)$, there exist $G$-invariant classes $\ell\in\NS(\Gamma),\gamma'_1 \in \NS^2(Y_1)$ such that we have the following commutative diagram:
\[
  \begin{CD}
     (\CH^2(Y_2))^{\gamma_2} @>{\overset{G}{\cong}}>> (\CH^2(Y_1))^{\gamma'_1}\times\Pic^{\ell}(\Gamma) \\
  @V{\text{transitive}}VV    @V{\text{transitive}}VV \\
     (\CH^2(Y_2))^{0}   @>{\overset{G}{\cong}}>>  (\CH^2(Y_1))^{0}\times\Pic^0(\Gamma)\\
  @V{\phi_{Y_2}^2}VV    @V{\phi_{Y_1}^2\times\phi^1_{\Gamma}}VV \\
     \IIJ(Y_2)(k)   @>{\overset{G}{\cong}}>>   \IIJ(Y_1)(k)\times \JJ(\Gamma)(k).
  \end{CD}
\]
Here, the horizontal morphisms are $G$-equivariant
and the lower morphism is an isomorphism as $G$-equivariant principally polarized abelian varieties. 
Then we obtain the following property.
\begin{itemize}
\item[(ii)]For any $G$-invariant class $\gamma_2\in \NS^2(Y_2)$, there exist $G$-invariant classes $\ell\in\NS(\Gamma),\gamma'_1 \in \NS^2(Y_1)$ and $G$-isomorphisms 
\[(\CCH^2_{Y_2/k})^{\gamma_2} \cong (\CCH^2_{Y_1/k})^{\gamma'_1}\times\PPic^{\ell}_{\Gamma/k},\]
\[(\CCH^2_{Y_2/k})^{0}   \cong  (\CCH^2_{Y_1/k})^{0}\times\PPic^0_{\Gamma/k}\]
such that we have the following commutative diagram:
\begin{equation*}
  \begin{CD}
     (\CCH^2_{Y_2/k})^{\gamma_2} @>{\overset{G}{\cong}}>> (\CCH^2_{Y_1/k})^{\gamma'_1}\times\PPic^{\ell}_{\Gamma/k} \\
  @V{\text{transitive}}VV    @V{\text{transitive}}VV \\
     (\CCH^2_{Y_2/k})^{0}   @>{\overset{G}{\cong}}>>  (\CCH^2_{Y_1/k})^{0}\times\PPic^0_{\Gamma/k}.\\
  \end{CD}
\end{equation*}
Here, the lower arrow is $G$-equivariant homomorphism as $G$-equivariant principally polarized abelian varieties.
\end{itemize}

Let $\gamma_2\in\NS^2(Y_2)$ and $m_1\in \NS(D_1)$ be any $G$-invariant classes and $C':=C\sqcup\Gamma$. 
By (i) and (ii), there exist $G$-invariant classes $\gamma_1\in \NS^2(Y_1), \ell\in \NS(\Gamma), n\in \NS(C), \delta \in \NS^2(X), n'=\NS(C')$ such that we have $G$-isomorphisms
\begin{align*}
(\CCH^2_{Y_2/k})^{\gamma_2}\times \PPic^{m_1}_{D_1/k}
&\cong(\CCH^2_{Y_1/k})^{\gamma_1}\times\PPic^{m_1}_{D_1/k}\times\PPic^{\ell}_{\Gamma/k}\\
&\cong(\CCH^2_{X/k})^{\delta}\times \PPic^{\ell}_{\Gamma/k}\times\PPic^{n}_{C/k}\\
&\cong(\CCH^2_{X/k})^{\delta}\times\PPic^{n'}_{C'/k}, 
\end{align*}
\[(\CCH^2_{Y_2/k})^{0} \times\PPic^{0}_{D_1/k}
\cong (\CCH^2_{X/k})^{0}\times\PPic^{0}_{C'/k},\]
where the second isomorphism is an isomorphism as $G$-equivariant principally polarized abelian varieties.
Moreover, we get the following commutative diagram:
\[
  \begin{CD}
     (\CCH^2_{Y_2/k})^{\gamma_2}\times \PPic^{m_1}_{D_1/k}
     @>{\overset{G}{\cong}}>> 
     (\CCH^2_{X/k})^{\delta}\times\PPic^{n'}_{C'/k} \\
  @V{\text{transitive}}VV    @V{\text{transitive}}VV \\
     (\CCH^2_{Y_2/k})^{0} \times\PPic^{0}_{D_1/k}
       @>{\overset{G}{\cong}}>>  
       (\CCH^2_{X/k})^{0}\times\PPic^{0}_{C'/k}.\\
  \end{CD}
\]
Therefore $Y_2$ satisfies the statement.

Similarly, if $Y_2$ satisfies the statement, then $Y_1$ also satisfies the statement.

\end{proof}

\begin{proof}[Proof of Theorem \ref{IJT}]
By Theorem \ref{important lemma}, there exist smooth projective $G$-curves $C,D$ such that, for any $G$-invariant classes $\gamma\in \NS^2(X), n\in \NS(C)$, we have a $G$-invariant class $m\in \NS(D) $, $G$-isomorphisms
\[(\CCH^2_{X/k})^{\gamma}\times \PPic^{n}_{C/k}\cong \PPic^{m}_{D/k}, \] 
\[(\CCH^2_{X/k})^{0}\times \PPic^{0}_{C/k}\cong \PPic^{0}_{D/k} \]
and the following commutative diagram:
\begin{equation}\label{maintheorem 1}
  \begin{CD}
     (\CCH^2_{X/k})^{\gamma_2}\times \PPic^{n}_{C/k} 
     @>{\overset{G}{\cong}}>>
     \PPic^m_{D/k} \\
  @V{\text{transitive}}VV    @V{\text{transitive}}VV \\
     (\CCH^2_{X/k})^{0}\times \PPic^{0}_{C/k}   @>{\overset{G}{\cong}}>>  \PPic^{0}_{D/k}.
  \end{CD}
\end{equation}
Here the lower arrow is an isomorphism as $G$-equivariant principally polarized abelian varieties. Note that the transitives preserve direct products.
Using Corollary \ref{decomposition of ppav} and the Torelli theorem for curves, there exists a smooth projective $G$-curve $D'$ such that $D=D' \sqcup C$. Then, by Corollary \ref{decomposition of ppav} and the diagram (\ref{maintheorem 1}), there exists a $G$-invariant class $m'\in \NS(D')$ such that
we have a $G$-isomorphism 
\[(\CCH^2_{X/k})^{\gamma}\cong \PPic^{m'}_{D'/k} \] 
as $G$-varieties and 
a $G$-isomorphism
\[(\CCH^2_{X/k})^{0}\cong \PPic^{0}_{D'/k} \]
as $G$-equivariant principally polarized abelian varieties.
\end{proof}

\begin{thm}(\cite[Theorem 3.5]{1})\label{CTT24 Theorem 3.4}
Let $X$ and $Y$ be smooth projective rationally connected threefolds with regular $G$-actions such that their intermediate Jacobians are irreducible as $G$-equivariant principally polarized abelian varieties. Assume that
\begin{itemize}
\item the $G$-equivariant IJT-obstruction does not vanish for $X$,
\item $\IIJ(X)$ and $\IIJ(Y)$ are not isomorphic as $G$-equivariant principally polarized abelian varieties
\end{itemize}
Then $X$ is not $G$-equivariantly birational to $Y$.
\end{thm}

\begin{proof}

Assume that $X$ and $Y$ are $G$-equivariantly birational. 
By the first assumption, there exists a $G$-invariant class $\gamma\in\NS^2(X)$ such that, for any smooth projective $G$-curve $C$ and $G$-invariant class $n$, the $G$-variety $(\CCH^2_{X/k})^{\gamma}$ is not isomorphic to $\PPic^n_{C/k}$. 
By Theorem \ref{important lemma}, there exist smooth projective $G$-curves $C',D$ and $G$-invariant classes $m\in \NS(D)$, $\gamma'\in \NS^2(Y)$ such that the following diagram is commutative:
\[
  \begin{CD}
     (\CCH^2_{X/k})^{\gamma}\times \PPic^0_{C'/k} @>{\overset{G}{\cong}}>> (\CCH^2_{Y/k})^{\gamma'}\times\PPic^m_{D/k} \\
  @V{\text{transitive}}VV    @V{\text{transitive}}VV \\
    (\CCH^2_{X/k})^{0}\times \PPic^0_{C'/k}   @>{\overset{G}{\cong}}>>  (\CCH^2_{Y/k})^{0}\times\PPic^0_{D/k}
  \end{CD}
\] 
Here the transitives preserve the direct products. There exists smooth projective $G$-irreducible curves $D_1,...,D_N$ such that $D$ and $D_1\sqcup \cdots\sqcup D_N$ are $G$-equivariant isomorphic. Then we have an isomorphism 
\[(\CCH^2_{Y/k})^{0}\times\PPic^0_{D/k}\overset{G}{\cong} (\CCH^2_{Y/k})^{0}\times\PPic^0_{D_1/k}\times \cdots \times \PPic^0_{D_N/k}\] 
as $G$-equivariant principally polarized abelian varieties. In addition, $G$-equivariant principally polarized abelian varieties $\PPic^0_{D_1/k}$,..., $\PPic^0_{D_N/k}$, and $(\CCH^2_{Y/k})^{0}$ are irreducible. Since $\IIJ(X)=(\CCH^2_{X/k})^{0}$ is irreducible, by Corollary \ref{decomposition of ppav} and the second assumption, we get an isomorphism 
\[(\CCH^2_{X/k})^{0}\cong \PPic^0_{D_{i}/k}\]
as $G$-equivariant principally polarized abelian varieties for some $i$. Then there exists a $G$-invariant class $m$ such that 
the following diagram is commutative:
\[
  \begin{CD}
     (\CCH^2_{X/k})^{\gamma} @>{\overset{G}{\cong}}>> \PPic^m_{D_i/k} \\
  @V{\text{transitive}}VV    @V{\text{transitive}}VV \\
    (\CCH^2_{X/k})^{0}   @>{\overset{G}{\cong}}>>  \PPic^0_{D_i/k}.
  \end{CD}
\] 
This contradicts the first assumption.
\end{proof}

\begin{cor}(\cite[Corollary 3.6]{1})
Let $X$ and $Y$ be smooth projective rationally connected threefolds with regular $G$-actions such that their intermediate Jacobians $\IIJ(X)$ and $\IIJ(Y)$ are isomorphic as $G$-equivariant principally polarized abelian varieties to the Jacobians of $G$-irreducuble smooth projective curves $C$ and $D$ whose connected components have genus $g\ge 2$, respectively. Assume that
\begin{itemize}
\item the $G$-equivariant IJT-obstruction does not vanish for $X$,
\item $C$ is not $G$-equivariantly isomorphic to $D$.
\end{itemize}
Then $X$ is not $G$-equivariantly birational to $Y$.
\end{cor}
\begin{proof}
By the second assumption and the Torelli theorem, $\JJ(C)$ is not $G$-equivariantly isomorphic to $\JJ(D)$ as $G$-equivariant principally polarized abelian varieties. The assertion follows from Theorem \ref{CTT24 Theorem 3.4}.

\end{proof}

\section{Automorphism groups of smooth Fano threefolds of \textnumero 2.18}\label{section 4}

In this section, we work over an algebraically closed field $k$ of characteristic zero.  The purpose of the section is to study (the orders of) the automorphism groups of Fano threefolds of \textnumero 2.18 (\cite{IP} and \cite{MM81}). \cite{CTT} first computed automorphism groups of some Fano threefolds of \textnumero 2.18.

These varieties have three structures: a conic bundle, a double cover, and a quadric surface bundle. Here a double cover is a finite flat morphism of degree $2$ over $k$.

\subsection{Fano threefolds of \textnumero 2.18}\label{section 4.3}
In the section, we present a method for computing the order of the automorphism groups of double covers of $\Line^1\times\Line^2$ ramified along a smooth divisor of bidegree $(2,2)$.

Let 
\[\pi:X\to\Line^1_{[t_0:t_1]}\times\Line^2_{[x:y:z]}\] 
be a double cover branched in a $(2,2)$-surface $Z$ over $k$. We call a prime $(2,2)$-divisor $Z$ on $\Line^1\times\Line^2$ a $(2,2)$-surface. A $(2,2)$-surface $Z$ is given by
\[t_0^2Q_1+2t_0t_1Q_2+t_1^2Q_3=0\]
for three quadric forms $Q_1,Q_2,Q_3\in k[x,y,z]$.
We denote by $\pi_1$ and $\pi_2$ the composition of the first and second projections in $\pi$, respectively. Then $X$ is smooth if and only if $Z$ is (see \cite[Lemma 2.51]{KM98}). The smooth projective variety $X$ is a smooth Fano threefold of \textnumero 2.18 in the Mori-Mukai classification. We consider smooth Fano threefolds of \textnumero 2.18, i.e., we assume that the branch locus $Z$ is smooth. 
The map $\pi_2:X\to\Line^2$ is a standard conic bundle on $\Line^2$, and the discriminant curve of $\pi_2$ is 
\[\Delta:Q_2^2-Q_1Q_3=0\subset \Line^2.\]
The subscheme $\Delta$ is reduced (see \cite[Section 2.1]{DJK24} or \cite[Corollary 1.9]{Sar82}). But the quartic curve $\Delta$ is not necessarily irreducible (see Example \ref{No.2.18 with reducible quartic}). By \cite[5.6.1. Corollary]{Pro18}, the conic bundle $X$ on $\Line^2$ with a quartic plane curve is rational. 
A point $p\in \Delta(k)$ is singular if and only if $p\in \{Q_1=Q_2=Q_3=0\}$(\cite[Lemma 2.15]{DJK24}). The quartic curve $\Delta$ has at worst $A_1$-singularities (see \cite[Section 2,1]{DJK24} or \cite[3.3.3. Corollary]{Pro18}). However, not every quartic curve with at worst $A_1$-singularities arises as the discriminant of a smooth Fano threefold of \textnumero 2.18.  The quartic plane curve $\Delta=\{x(zy^2-x^3+az^2x+bz^3)=0\}\subset \Line^2_{[x:y:z]}$ over $\C$ is one example of such curves, where $4a^3\neq 27b^2$, $a\neq 0$, and $b\neq 0$. Indeed, since the curve $\Delta$ has a decomposition $\Delta=H\cup C$ such that $H$ and $C$ intersect distinct three points, by \cite[(3.8.2) and Section 7]{Pro18}, $\Delta$ does not arise as the discriminant of a smooth Fano threefold of \textnumero 2.18.

The morphism 
\[\pi_2|_{Z\setminus\pi_2^{-1}(\Sing\Delta)}:Z\setminus\pi_2^{-1}(\Sing\Delta)\to\Line^2\setminus\Sing\Delta\]
is a double cover. There exists uniquely a non-trivial birational involution $\sigma$ on $Z$ acting trivially on $\Line^2$. We say that a smooth $(2,2)$-divisor $Z$ has the non-trivial involution on $Z$ acting trivially on $\Line^2$ if the involution $\sigma$ is regular.

If $\Delta$ is smooth, we see that $Z$ is a smooth del Pezzo surface of degree $2$ with double cover $\pi_2|_Z:Z\to \Line^2$ branched in $\Delta$. Conversely, a smooth del Pezzo surface is a smooth $(2,2)$-divisor on $\Line^1\times\Line^2$, and the projection to $\Line^2$ is a double cover branched in a smooth plane quartic. When $\Delta$ is smooth, we have a relation between automorphism groups of $Z$ and $\Delta$.
\begin{thm}\label{aut of del pezzo}(\cite[Section 8.7.3]{CAG})
Let $Z$ be a smooth del Pezzo surface of degree $2$ with the smooth plane quartic $\Delta$ over $k$. Then $\Aut(Z)$ is isomorphic to an extension of $\Aut(\Delta)$ by the cyclic group of order $2$, generated by the non-trivial involution acting trivially on $\Line^2$. 
\end{thm}
In addition, the following theorem is a property of automorphism groups of smooth plane quartic curves:
\begin{thm}(\cite[Lemma 6.5.1]{CAG})\label{cyclic aut of quartic}
Let $\Delta\subset \Line^2$ be a smooth plane quartic curve over $k$ and $\varphi$ is an automorphism of $\Delta$ of order $n$. Then, by choosing coordinates, the generator of the cyclic group $\langle \varphi \rangle$ is represented by the diagonal matrix $\diag(1, \zeta_n^a, \zeta_n^b)$, where $\zeta_n$ is a primitive $n$-th root of unity.

The cyclic action on $\Delta$, represented by $\diag(1,\zeta_n^a,\zeta_n^b)$, is called Type $n;(a,b)$. 
\end{thm}

The following proposition means that given any three quadric forms $Q_1,Q_2,Q_3
\in k[x,y,z]$ such that $Q_2^2-Q_1Q_3=0$ is a smooth plane quartic, we obtain a smooth Fano threefold of \textnumero 2.18.
\begin{prop}(\cite[Lemma 2.15 (3)]{DJK24})\label{smoothness}
For three quadric forms $Q_1,Q_2,Q_3\in k[x,y,z]$ such that $Q_2^2-Q_1Q_3=0$ is a smooth plane quartic, the surface 
\[t_0^2Q_1+2t_0t_1Q_2+t_1^2Q_3=0 \subset \Line^1_{[t_0:t_1]}\times\Line^2_{[x:y:z]}\]
is smooth.
\end{prop}

\subsubsection*{A relation between $\Aut(X)$ and $\Aut(Z)$}
We introduce some tools for computing the automorphism groups of $X$. The following claim is well known to experts. 
\begin{prop}(cf. \cite[Proposition 9.3 Case 1]{DJK24})\label{two dc}
Let $\pi:\mathcal{X}\to \Line^1\times \Line^2$ and $\pi':\mathcal{X}'\to \Line^1\times \Line^2$ be double covers branched in a smooth $(2,2)$-divisor. Assume that we have an isomorphism $\phi:\mathcal{X}\cong \mathcal{X}'$. Then for $i=1,2$ there exists an automorphism $\psi_i$ of $\Line^i$ such that the following diagram is commutative:
\[
  \begin{CD}
     \mathcal{X} @>{p_i\circ\pi}>> \Line^i \\
  @V{\phi}VV    @V{\psi_i}VV \\
      \mathcal{X}'  @>{p_i\circ\pi'}>>  \Line^i,
  \end{CD}
\]
where $p_1:\Line^1\times\Line^2\to\Line^1$ and $p_2:\Line^1\times\Line^2\to\Line^2$ are the first and second projection, respectively. 
\end{prop}
\begin{proof}
Following \cite[Proposition 9.3 Case 1]{DJK24}, we prove the claims. We have extremal rays $R_1$, $R_2$ such that the contraction morphisms associated to $R_1$ and $R_2$ are $p_1\circ\pi$ and $p_2\circ\pi$, respectively. Since the Picard number $\rho(\mathcal{X})=2$, the extremal faces of the Mori cone $\overline{\NE(\mathcal{X})}$ are only $R_1$, $R_2$. The morphisms $p_1\circ\pi'\circ\phi:\mathcal{X}\to \Line^1$ is a quadric surface bundle on $\Line^1$. Suppose that $\NE(p_1\circ\pi'\circ\phi)=R_2=\NE(p_2\circ\pi)$. Since $p_1\circ\pi'\circ\phi, p_2\circ\pi$ are contraction morphisms, we have these are isomorphic. This is a contraction. 
Hence we see that $\NE(p_1\circ\pi'\circ\phi)=R_1=\NE(p_1\circ\pi)$, i.e. there is an isomorphism $\psi_1$ such that we have the following commutative diagram: 
\[
  \begin{CD}
     \mathcal{X} @>{p_1\circ\pi}>> \Line^1 \\
  @V{\phi}VV    @V{\psi_1}VV \\
      \mathcal{X}'  @>{p_1\circ\pi'}>>  \Line^1.
  \end{CD}
\]
As $p_2\circ\pi'\circ\phi:\mathcal{X}\to \Line^2$ is a conic bundle on $\Line^2$, a similar argument shows that there exists an automorphism $\psi_2$ such that $\psi_1\circ p_2\circ\pi=p_2\circ\pi'\circ\phi$.

\end{proof}
By the above proposition, the morphisms $\pi,\pi_1$ and $\pi_2$ are $\Aut(X)$-equivariant. 
The following claim is stated over $\C$ (\cite{K24}), but still holds over $k$ :
\begin{prop}(\cite[Proposition 1.47]{K24} and \cite[Lemma 4.4.1]{KPS18})
Let $\varphi:\mathcal{X}\to \mathcal{Y}$ be a finite surjective morphism of degree $2$ over $k$ and $\tau:\mathcal{X}\to \mathcal{X}$ be the non-trivial involution that trivially acts on $\mathcal{Y}$. Assume that
\begin{itemize}
\item $\mathcal{X},\mathcal{Y}$ are normal varieties,
\item the Weil divisor class group $\Cl(\mathcal{Y})$ has no non-trivial 2-torsion, and 
\item the morphism $\varphi$ is $\Aut(\mathcal{X})$-equivariant. 
\end{itemize}
Then we get the exact sequence 
\[1\to \langle \tau \rangle\to \Aut(\mathcal{X})\overset{\bar{\varphi}}{\to} \Aut(\mathcal{Y}; \mathcal{Z})\to 1,\]
where $\mathcal{Z}$ is the ramification locus of $\varphi$, $\Aut(\mathcal{Y}; \mathcal{Z})\subset \Aut(\mathcal{Y})$ is the subgroup which preserves $\mathcal{Z}$, and $\bar{\varphi}$ is the natural group homomorphism induced by $\varphi$. 
\end{prop}

\begin{prop}\label{aut of 2.18}
Let $\pi:X\to\Line^1\times\Line^2$
be a double cover branched in a smooth $(2,2)$-divisor $Z$ over $k$, and 
$\tau:X\to X$ be the non-trivial involution that trivially acts on $\Line^1\times\Line^2$. Then we have the exact sequence of groups
\[1\to \langle \tau \rangle\to \Aut(X)\overset{\bar{\pi}}{\to} \Aut(\Line^1\times\Line^2; Z)\to 1.\]
\end{prop}
\begin{proof}
By Proposition \ref{two dc}, the double cover $\pi:X\to \Line^1\times \Line^2$ is $\Aut(X)$-equivariant. Therefore the assertion holds. 
\end{proof}

\begin{prop}(\cite[Lemma 8.3 and Lemma 2.1]{PCS19})\label{PCS19 Lemma}
Let $\pi:X\to\Line^1\times\Line^2$
be the double cover whose branched locus $Z$ is a smooth $(2,2)$-divisor over $k$.
Then the natural map
\[\Aut(\Line^1\times\Line^2; Z)\to \Aut(Z)\]
is injective.
\end{prop}

To calculate to $\Aut(X)$, it is necessary to check when automorphisms of $Z$ lift to $\Line^1\times\Line^2$. 

\subsubsection*{Subgroups of $\Aut (\Line^1\times\Line^2) $ which preserves the branch locus $Z$}
From the above propositions, in order to compute the automorphism group of $X$, we need to compute $\Aut (\Line^1\times\Line^2; Z) $. An example of a non-trivial automorphism of $Z$ is a non-trivial involution $\sigma$ acting trivially on $\Line^2$. It is unclear whether the above involution $\sigma$ always exists, but if it does exist, it is unique. When $\Delta$ is smooth, we have the non-trivial involution associated to the double cover $Z\to \Line^2$, and it acts trivially on $\Line^2$. In this section, we consider the problem of when the non-trivial involution $\sigma$ on $Z$ that acts trivially on $\Line^2$ lifts to $\Line^1\times\Line^2$.

\begin{lem}\label{involution lemma}
Assume that a smooth $(2,2)$-surface $Z$ has the non-trivial involution $\sigma:Z\to Z$ acting trivially on $\Line^2$. The preimage of $\Line^2\setminus\{Q_1Q_3=0\}$ under the map $Z\to \Line^2$ is denoted by $U\subset Z$. We define 
\begin{align*}
\sigma':U\to U
;([t_0:t_1], [x:y:z])\mapsto ([Q_3t_1:Q_1t_0], [x:y:z]). 
\end{align*}
Then $\sigma=\sigma'$ on $U$.
\end{lem}
\begin{proof}
For $([t_0:t_1],[x:y:z])\in U$, the equation
\[Q_1(t_1Q_3)^2+2Q_2(t_1Q_3)(t_0Q_1)+Q_3(t_0Q_1)^2=Q_1Q_3(Q_3t_1^2+2Q_2t_0t_1+Q_1t_0^2)=0\]
holds. Hence $\sigma'$ is well-defined. 

We have $\sigma'^2=\id$ on $U$, and $\langle \sigma'\>$ acts trivially on $\Line^2\setminus\{Q_1Q_3=0\}$. Suppose that $\sigma'$ is trivial. Then we have $Q_3t_1^2=t_0^2Q_1$. Since we get
\[Q_1t_0^2+Q_2t_0t_1=Q_2t_0t_1+Q_3t_1^2=0,  \]
the equation
\[t_0^2t_1^2(Q_2^2-Q_1Q_3)=0\]
holds on $U$. Hence we see that $\Line^2\setminus\{Q_1Q_3=0\}\subset\Delta$. Therefore the contradiction.
\end{proof}

\begin{thm}\label{lifting theorem}
The following are equivalent: 
\begin{itemize}
\item A smooth $(2,2)$-surface $Z$ has the non-trivial involution $\sigma$ acting trivially on $\Line^2$ and $\sigma$ lifts to $\Line^1\times\Line^2$.

\item \begin{itemize}
\item[(i)]There exist $a,b\in k$ such that $2Q_2=aQ_1+bQ_3$ and $ab\neq 1$, or
\item[(ii)]there exists $a\in k^*$ such that $Q_3=aQ_1$.
\end{itemize}
\end{itemize}

Moreover, if $Z$ has the non-trivial involution $\sigma$ acting trivially on $\Line^2$ and $\sigma$ lifts to $\Line^1\times\Line^2$, then we get 
\[\sigma=
\begin{cases}
\begin{pmatrix}
1&a\\
-b&-1
\end{pmatrix}\times \id_{\Line^2}&\text{if (i)}\\
\begin{pmatrix}
0&a\\
1&0
\end{pmatrix}\times \id_{\Line^2}&\text{if (ii).}
\end{cases}\]

\end{thm}

\begin{proof}
Assume that $Z$ has the non-trivial involution $\sigma$ acting trivially on $\Line^2$ and $\sigma$ lifts to $\Line^1\times\Line^2$. 
Then there exists a matrix $M=\begin{pmatrix}a_{11} &a_{12}\\a_{21}&a_{22}\end{pmatrix}\in \GL_2(k)$ such that $\sigma=M\times \id$. By Lemma \ref{involution lemma}, we have $\sigma'=M\times \id$ on $Z\setminus\pi_2^{-1}(\{Q_1Q_3=0\})$. where $\sigma'$ is the involution in Lemma \ref{involution lemma}. Hence, $Z$ is contained in 
\[Z':=\{Q_1t_0(a_{11}t_0+a_{12}t_1)=Q_3t_1(a_{21}t_0+a_{22}t_1)\}.\]
Since $\sigma'=M\times \id$ on $Z\setminus\{Q_1Q_3=0\}$, for any $([t_0:t_1], [x:y:z])\in Z\setminus\{Q_1Q_3=0\}$, we have $[a_{11}t_0+a_{12}t_1:a_{21}t_0+a_{22}t_1]=[Q_3t_1:Q_1t_0]$, and 
\[Q_1t_0(a_{11}t_0+a_{12}t_1)=Q_3t_1(a_{21}t_0+a_{22}t_1).\]
Hence we have 
\[Z\setminus\{Q_1Q_3=0\}\subset Z'=\{Q_1t_0(a_{11}t_0+a_{12}t_1)=Q_3t_1(a_{21}t_0+a_{22}t_1)\}.\]
As $Z$ is irreducible, we get $Z\subset Z'$.

If $Z'\subsetneq \Line^1\times\Line^2$, then there exists $c\in k^*$ such that
\begin{align*}
&Q_1t_0(a_{11}t_0+a_{12}t_1)-Q_3t_1(a_{21}t_0+a_{22}t_1)\\
&=c(Q_1t_0^2+2Q_2t_0t_1+Q_3t_1^2)
\end{align*}
in $ k[x,y,z]\otimes k[t_0,t_1]$. Set $a:=\frac{a_{12}}{c}, b=-\frac{a_{21}}{c}$. Then we get 
\[2Q_2=aQ_1+bQ_3\]
and 
\[\sigma=\begin{pmatrix}
c&ca\\
-cb&-c
\end{pmatrix}\times \id_{\Line^2}=\begin{pmatrix}
1&a\\
-b&-1
\end{pmatrix}\times \id_{\Line^2}.
\]
In addition, we obtain $ab\neq 1$, $\sigma\neq \id$, and $\sigma^2=\id$.
If $Z'=\Line^1\times\Line^2$, we obtain $Q_3=aQ_1$ and 
\[\sigma=\begin{pmatrix}
0&a_{12}\\
a_{21}&0
\end{pmatrix}\times \id_{\Line^2}=\begin{pmatrix}
0&a\\
1&0
\end{pmatrix}\times \id_{\Line^2},\]
where $a=\frac{a_{12}}{a_{21}}$.

Assume that (i) holds. Let $\sigma\in \Aut(\Line^1\times\Line^2)$ be the non-trivial involution
\[\begin{pmatrix}
1&a\\
-b&-1
\end{pmatrix}\times \id_{\Line^2}.\]
By $2Q_2=aQ_1+bQ_3$, we have $\sigma(Z\cap\{Q_1=Q_3=0\})\subset Z$. 
For any $P=([t_0:t_1], [x:y:z])\in Z\setminus\pi_2^{-1}(\{Q_1Q_3=0\})$, we have 
\[\sigma(P)=([Q_3t_1:Q_1t_0], [x:y:z]).\]
Hence, $\sigma(Z\setminus \{Q_1Q_3= 0\})\subset Z$. Since $Z\cap\{Q_3=0\}\cap\{Q_1\neq 0\}$ is contained in 
\[\{([0:1],[x:y:z]), ([-a:1],[x:y:z])|Q_3(x,y,z)=0\}\subset Z,\]
and we have
\[\sigma([0:1],[x:y:z])=([-a:1],[x:y:z]),\]\[ \sigma([-a:1],[x:y:z])=([0:1],[x:y:z]), \]
we get $\sigma(Z\cap\{Q_3=0\}\cap\{Q_1\neq 0\})\subset Z$. 
Similarly, we obtain $\sigma(Z\cap\{Q_1=0\}\cap\{Q_3\neq 0\})\subset Z$. Therefore we have a non-trivial involution $\sigma$ on $Z$, $\sigma$ acts trivially on $\Line^2$, and $\sigma$ lifts to $\Line^1\times\Line^2$.
If (ii) holds, then the non-trivial involution 
\[\sigma=\begin{pmatrix}
0&a\\
1&0
\end{pmatrix}\times \id_{\Line^2}\]
on $\Line^1\times\Line^2$ induces a non-trivial involution $\sigma$ on $Z$. The isomorphism $\sigma$ acts trivially on $\Line^2$, and $\sigma$ lifts to $\Line^1\times\Line^2$. Therefore, if (i) or (ii) hold, then $Z$ has a non-trivial involution $\sigma$ acting trivially on $\Line^2$ and $\sigma$ lifts to $\Line^1\times\Line^2$.
\end{proof}

\begin{cor}\label{involution does not lift}
Assume that a smooth $(2,2)$-surface $Z$ has the non-trivial involution $\sigma$ on $Z$, acting trivially on $\Line^2$. If $\Delta$ is irreducible, $\sigma$ does not lift to $\Line^1\times \Line^2$. In particular, when $\Delta$ is smooth, the non-trivial involution $\sigma$ associated to the double cover $Z\to \Line^2$ does not lift to $\Line^1\times \Line^2$.
\end{cor}
\begin{proof}
Note that the quartic curve $\Delta:Q_2^2-Q_1Q_3=0\subset \Line^2$ is reduced. 
Assume that a smooth $(2,2)$-surface $Z$ has the non-trivial involution $\sigma$ acting trivially on $\Line^2$ and $\sigma$ lifts to $\Line^1\times\Line^2$. By Theorem \ref{lifting theorem}, quadric forms $Q_1,Q_2$, and $Q_3$ satisfy one of the followings:
\begin{itemize}
\item[(i)]There exist $a,b\in k$ such that $2Q_2=aQ_1+bQ_3$ and $ab\neq 1$, or
\item[(ii)]there exists $a\in k^*$ such that $Q_3=aQ_1$.
\end{itemize}
We consider the case (i). When $a=0$, we get 
\[Q_2^2-Q_1Q_3=Q_3\left(\frac{b^2}{4}Q_3-Q_1\right).\]
When $a\neq 0$, we get 
\[Q_2^2-Q_1Q_3=\frac{a^2}{4}(Q_1-s_1Q_3)(Q_1-s_2Q_3),\]
for some $s_1,s_2\in k$. Therefore $\Delta$ is reducible.

We consider the case (ii). Then $Q_2^2-Q_1Q_3=(Q_2-\sqrt{a}Q_1)(Q_2+\sqrt{a}Q_1)$, and $\Delta$ is reducible. 
\end{proof}

\subsubsection*{Examples of smooth $(2,2)$-surfaces with a non-trivial involution acting trivially on $\Line^2$}
Let $\Delta$ be a quartic plane curve. If $\Delta$ has $4$ singular points which are $A_1$ and $\Delta$ consists of two non-degenerate conics, by coordinate changing, $\Delta$ is given by
 \[(yz+xz+xy)(\alpha yz+\beta xz +xy)=0, \]
 where $\alpha\beta(\alpha-1)(\beta-1)(\alpha-\beta)\neq 0$ (see \cite{Hui79} and \cite[Table 2]{vBDL25}). The singular points of $\Delta$ are \[P_1=[1:0:0], P_2=[0:1:0], P_3=[0:0:1], P_4=\left[-\frac{\beta-\alpha}{\beta-1}:\frac{\beta-\alpha}{\alpha-1}:1\right].\]
We claim that there exists a smooth Fano threefold $X$ of \textnumero 2.18 whose discriminant curve is a quartic plane curve $\Delta$ such that a smooth $(2,2)$-surface has the non-trivial involution which acts trivially on $\Line^2$ and lifts to $\Line^1\times\Line^2$.

\begin{example}\label{No.2.18 with reducible quartic}
Assume that $a,b,\alpha, \beta, \lambda, s\in k^*$ satisfy
\begin{itemize}
\item $ab\neq 1, 2$,
\item $s^2+\left(2-\frac{4}{ab}\right)s+1=0$,
\item $(\alpha-1)(\beta-1)(\alpha-\beta)\neq 0$.
\end{itemize}
Set
\[\begin{pmatrix}
Q_1'\\Q_3'
\end{pmatrix}=\frac{1}{\lambda(1-s^2)}\begin{pmatrix}
1&-s^2\\
s&-s
\end{pmatrix}\begin{pmatrix}
\lambda^2(yz+xz+xy)\\
\alpha yz+\beta xz+xy\end{pmatrix},\]
\[Q_1=\frac{2}{a}Q_1', Q_2=Q_1'+Q_3', Q_3=\frac{2}{b}Q_3'.\]
Then 
\[Z:t_0^2Q_1+2t_0t_1Q_2+t_1^2Q_3=0\subset \Line^1_{[t_0:t_1]}\times \Line^2_{[x:y:z]}\]
is smooth, and 
\[Q_2^2-Q_1Q_3=(yz+xz+xy)(\alpha yz+\beta xz +yx).\]
Moreover, $Z$ has the non-trivial involution
\[\sigma=\begin{pmatrix}
1&a\\
-b&-1
\end{pmatrix}\times \id_{\Line^2}.
\]

\end{example}

\begin{proof}
Since
\begin{align*}
\begin{pmatrix}
1&-s\\1&-\frac{1}{s}
\end{pmatrix}\begin{pmatrix}
Q_1'\\Q_3'
\end{pmatrix}
&=
\frac{1}{\lambda(1-s^2)}\begin{pmatrix}
1&-s\\1&-\frac{1}{s}
\end{pmatrix}\begin{pmatrix}
1&-s^2\\
s&-s
\end{pmatrix}\begin{pmatrix}
\lambda^2(yz+xz+xy)\\
\alpha yz+\beta xz+xy\end{pmatrix}\\
&=
\begin{pmatrix}
\lambda(yz+xz+xy)\\
\frac{1}{\lambda}(\alpha yz+\beta xz+xy)\end{pmatrix},
\end{align*}
we get
\begin{align*}
Q_2^2-Q_1Q_3&=Q_1'^2+\left(2-\frac{4}{ab}\right)Q_1'Q_3'+Q_3'^2\\
&=(Q_1'-sQ_3')\left(Q_1'-\frac{1}{s}Q_3'\right)\\
&=(yz+xz+xy)(\alpha yz+\beta xz +yx).
\end{align*}
We prove that $Z$ is smooth. By \cite[Lemma 2.15]{DJK24}, it is enough to show that $Z$ is non-singular on $\Line^1\times\{P_1,P_2,P_3,P_4\}$. Here \[P_1=[1:0:0], P_2=[0:1:0], P_3=[0:0:1], P_4=\left[-\frac{\beta-\alpha}{\beta-1}:\frac{\beta-\alpha}{\alpha-1}:1\right].\] 
Define 
\[F=t_0^2Q_1+2t_0t_1Q_2+t_1^2Q_3\in k[t_0:t_1]\otimes k[x:y:z],\]
\[p=\frac{1}{a}t_0^2+(s+1)t_0t_1+\frac{s}{b}t_1^2\in k[t_0:t_1],\]
\[q=\frac{s}{a}t_0^2+(s+1)t_0t_1+\frac{1}{b}t_1^2\in k[t_0:t_1].\]
By direct computation, we obtain the following table:\\
\begin{tabular}{|l|r|r|r|r|} \hline
   $ $&$P_1$&$P_2$&$P_3$&$P_4$\\ \hline \hline
$\frac{1}{2}\lambda(1-s^2)\frac{\partial F}{\partial x}$&$ $&$\lambda^2p-sq$&$\lambda^2p-s\beta q$&$\frac{\beta-1}{\alpha -1}(\lambda^2p-s\alpha q)$\\\hline
$\frac{1}{2}\lambda(1-s^2)\frac{\partial F}{\partial y}$&$\lambda^2p-sq$&$ $&$\lambda^2p-s\alpha q$&$\frac{\alpha-1}{\beta -1}(\lambda^2p-s\beta q)$\\\hline
$\frac{1}{2}\lambda(1-s^2)\frac{\partial F}{\partial z}$&$\lambda^2 p-s\beta q$&$\lambda^2p-s\alpha q$&$ $&$ $\\\hline
\end{tabular}\\
Suppose that $([t_0:t_1],P_i)$ is a singular point of $Z$ for some $i$. As $s\lambda^2(\alpha-1)(\beta-1)(\alpha-\beta)\neq 0$, we obtain $p=q=0$. Since $t_0, 1-s^2\neq 0$, and
\[0=p-qs=\frac{1-s^2}{a}t_0^2+(1-s^2)t_0t_1, \]
we get $[t_0:t_1]=[-a:1]$. By $p([-a:1])=0$, we have $ab=1$. Hence the contradiction.
\end{proof}

\subsubsection*{Examples of Fano threefolds of \textnumero 2.18 with smooth plane quartics}
We introduce properties and examples of automorphism groups of Fano threefolds of \textnumero 2.18 with smooth plane quartics. The examples presented in this section were considered in \cite{CTT}. In section, $\pi:X\to\Line^1\times\Line^2$ is a double cover branched a smooth $(2,2)$-divisor $Z$ over $k$, and $\Delta$ is discrminat curve of the conic bundle $\pi_2:X\to \Line^1\times\Line^2\to \Line^2$. 
\begin{prop}\label{iinjective}
Assume that $\Delta$ is smooth. Then the group morphism 
\[\Aut(\Line^1\times\Line^2; Z)\to \Aut(\Delta),\]
induced by $\pi_2$, is injective.
\end{prop}
\begin{proof}
Since $\Delta$ is smooth, the $(2,2)$-surface $Z$ has the non-trivial involution $\sigma:Z\to Z$ acts trivially on $\Line^2$.
Let $\alpha\in \Aut(\Line^1\times\Line^2;Z)$ be an automorphism acting trivially on $\Line^2$ (and $\Delta$). By Theorem \ref{aut of del pezzo}, we get $\alpha\in \langle\sigma\>$ via the injective $\Aut(\Line^1\times\Line^2;Z)\inj\Aut(Z)$ in Proposition \ref{PCS19 Lemma}. By Corollary \ref{involution does not lift}, the automorphism $\alpha$ is trivial. 
\end{proof}

\begin{cor}\label{upper of aut of 2.18}
Assume that $\Delta$ is smooth. 
Then there exists a subgroup $G\subset \Aut(\Delta)$ such that we have an exact sequence of groups 
\[0\to  \Z/2\Z\to \Aut(X)\to  G\to 1.\]

\end{cor}
The following examples are Fano threefolds of \textnumero 2.18 with smooth quartics for which $|\Aut(X)|=2\times |\Aut(\Delta)|$ holds.

\begin{example}
Let quadrics $Q_1,Q_2$, and $Q_3$ be $\sqrt{-1}(y^2+xz), x^2+ayz$ and $\sqrt{-1}(z^2+xy)$ respectively. Here $a\in k$. Assume that $Q_2^2-Q_1Q_3$ is of the form
\[x^4+x(y^3+z^3)+\alpha yzx^2+\beta y^2z^2\]
and $\Delta=\{Q_2^2-Q_1Q_3=0\}$ is smooth. 
Here $\alpha, \beta$ satisfy $\alpha\neq \beta$ and $\alpha\beta\neq 0$. For instance, when $a=1$, we get $\beta=2$, $\alpha=3$, and the assumptions hold. 
Then $\Aut(\Line^1\times\Line^2; Z)\cong \mathfrak{S}_3$. Therefore $|\Aut(X)|=12$.
\end{example}

\begin{proof}
By \cite[Theorem 6.5.2 Type IX]{CAG}(or \cite[Theorem 16]{B08}), we have \[\Aut(\Delta)\cong \mathfrak{S}_3.\]
The group $\Aut(\Line^1\times\Line^2; Z)$ contains the group, generated by
\[\left( 
\begin{pmatrix}
\zeta_3&0\\
0&1
\end{pmatrix},\begin{pmatrix}
1&0&0\\
0&\zeta_3&0\\
0&0&\zeta_3^2
\end{pmatrix}
\right)\text{\,and \,}
\left(\begin{pmatrix}
0&1\\
1&0
\end{pmatrix},\begin{pmatrix}
1&0&0\\
0&0&1\\
0&1&0
\end{pmatrix} \right),
\]
which is isomorphic to $\mathfrak{S}_3$. Hence $\Aut(\Line^1\times\Line^2; Z)\cong \mathfrak{S}_3$. Therefore $|\Aut(X)|=12$.
\end{proof}

\begin{example}
If $Q_1=\sqrt{-1}x^2+y^2, Q_2=z^2$ and $Q_3=\sqrt{-1}x^2-y^2$, then we have 
\[|\Aut(X)|=2\times |\Aut(\Delta)|=192.\]
\end{example}

\begin{proof}
From $Q_2^2-Q_1Q_3=x^4+y^4+z^4$ and \cite[Theorem 6.5.2 Type II]{CAG}(or \cite[Theorem 16]{B08}), we have 
\[\Aut(\Delta)\cong (\Z/4\Z\times \Z/4\Z)\rtimes\mathfrak{S}_3.\]
The group $\Aut(\Delta)$ is generated by
\[\begin{pmatrix}
\sqrt{-1}&0&0\\
0&1&0\\
0&0&1
\end{pmatrix}, 
\begin{pmatrix}
1&0&0\\
0&\sqrt{-1}&0\\
0&0&1
\end{pmatrix},
\begin{pmatrix}
0&1&0\\
1&0&0\\
0&0&1
\end{pmatrix},\text{and}
\begin{pmatrix}
0&0&1\\
1&0&0\\
0&1&0
\end{pmatrix}.\]
Indeed, we have $\Aut(\Delta)\cong (\Z/4\Z\times \Z/4\Z)\rtimes\mathfrak{S}_3$ (See \cite[Theorem 16]{B08} or \cite[Theorem 6.5.2]{CAG}). The generators give a subgroup of $\Aut(\Line^2)$ of order $96$ and preserve the Fermat quartic $\Delta$.

Consider the automorphisms 
\[\tau_1:([t_0:t_1], [x:y:z])\mapsto \left(\left[-t_1:t_0\right], [\sqrt{-1}x:y:z]\right),\]
\[\tau_2:([t_0:t_1], [x:y:z])\mapsto ([t_1:t_0], [x:\sqrt{-1}y:z]),\]
\[\tau_3:([t_0:t_1], [x:y:z])\mapsto ([\sqrt{-1}t_1:t_0],[y:x:z]),\]
\[\tau_4:([t_0:t_1], [x:y:z])\mapsto ([t_1-t_0:\sqrt{-1}(t_0+t_1)],[z:x:y]),\]
on $\Line^1\times\Line^2$. Then $\tau_i$'s preserve $Z$, and the images of $\tau_i$'s by the map $\Aut(\Line^1\times\Line^2; Z)\to \Aut(\Delta)$ are generators of $\Aut(\Delta)$. By Proposition \ref{iinjective}, we get 
\[\Aut(\Line^1\times\Line^2; Z)\cong \Aut(\Delta)\cong(\Z/4\Z\times \Z/4\Z)\rtimes\mathfrak{S}_3).\]
\end{proof}

The following example is a Fano threefold of \textnumero 2.18 with a smooth quartic for which $|\Aut(X)|\neq 2\times |\Aut(\Delta)|$ holds. 

\begin{example}(\cite{Mue23})
If $Q_1=x^2-xy+y^2+xz, Q_2=x^2-xy+yz$, and $Q_3=x^2-xz+yz+z^2$, then we have 
\[|\Aut(X)|<2\times |\Aut(\Delta)|.\]
\end{example}

\begin{proof}
The Klein quartic
\[\Delta=\{Q_1Q_3-Q_2^2=x^3y+y^3z+z^3x=0\}\]
has the automorphism
\[M=\begin{pmatrix}
1&0&0\\
0&\zeta^3&0\\
0&0&\zeta
\end{pmatrix}\]
of order $7$, where $\zeta\in \C\setminus{1}$ such that $\zeta^7=1$. If the natural injection $\Aut(\Line^1\times\Line^2; Z)\to \Aut(\Delta)$ is surjective, then there exists a matrix 
\[N=\begin{pmatrix}
a_{00}&a_{01}\\
a_{10}&a_{11}
\end{pmatrix}\in \GL_2(k)\]
such that 
\begin{align*}
t_0^2&(x^2-xy+y^2+xz)+2t_0t_1(x^2-xy+yz)+t_1^2(x^2-xz+yz+z^2)\\
=&(a_{00}t_0+a_{01}t_1)^2(x^2-\zeta^3xy+\zeta^6y^2+\zeta xz)\\
&+2(a_{00}t_0+a_{01}t_1)(a_{10}t_0+a_{11}t_1)(x^2-\zeta^2xy+\zeta^4yz)\\
&+(a_{10}t_0+a_{11}t_1)^2(x^2-\zeta xz+\zeta^4yz+\zeta^2z^2).
\end{align*}
Comparing the coefficients of $y^2$ on both sides, we get $\zeta^6a_{00}^2=\zeta^2a_{11}^2=1, a_{01}=a_{10}=0$. Thus, we have 
\begin{align*}
t_0^2&(x^2-xy+y^2+xz)+2t_0t_1(x^2-xy+yz)+t_1^2(x^2-xz+yz+z^2)\\
=&a_{00}^2t_0^2(x^2-\zeta^3xy+\zeta^6y^2+\zeta xz)+2a_{00}a_{11}t_0t_1(x^2-\zeta^2xy+\zeta^4yz)\\
&+a_{11}^2t_1^2(x^2-\zeta xz+\zeta^4yz+\zeta^2z^2).
\end{align*}
Comparing the coefficient of $x^2t_0t_1$, we obtain $a_{00}a_{11}=1$. This contradicts $\zeta^6a_{00}^2=\zeta^2a_{11}^2=1$. 
\end{proof}

\subsection{Family of smooth Fano threefolds of \textnumero 2.18}
In this section, we introduce a variety $M_{(2,2)}^{sm}$ that parametrizes isomorphism classes of smooth Fano threefolds of \textnumero 2.18. We prove that the automorphism group of a Fano variety, corresponding to a general point of $M_{(2,2)}^{sm}$, is isomorphic to $\Z/2\Z$. We consider the GIT-quotient 
\[|\ShO_{\Line^1\times\Line^2}(2,2)|^{ss}/\!\!/\SL_2\times \SL_3,\]
which was studied in \cite{PR20} and \cite{DJK24}. Also, \cite{DJK24} describes the compact K-moduli space of Fano threefolds in the deformation family of \textnumero 2.18. 

We have a natural $\SL_2(k)\times \SL_3(k)$-action on the linear system $|\ShO_{\Line^1\times\Line^2}(2,2)|$. As $\SL_2(k)\times \SL_3(k)$ is reductive, the GIT-quotient 
\[M_{(2,2)}:=|\ShO_{\Line^1\times\Line^2}(2,2)|^{ss}/\!\!/\SL_2\times \SL_3\]
 is defined, and its dimension equals to 6 (see \cite[Section 1]{PR20}). 

\cite[Lemma 2.20]{DJK24} is stated over $\C$, but still holds over $k$. Hence, if a quartic curve $\Delta=\{Q_2^2-Q_1Q_3=0\}$ is semi-stable (resp. stable), then a $(2,2)$-surface 
\[Z:t_0^2Q_1+2t_0t_1Q_2+t_1^2Q_3=0\]
is semi-stable (resp. stable), where $Q_1,Q_2,Q_3$ are quadric forms. In particular, a smooth $(2,2)$-surface is stable. Hence, we see that an open locus $|\ShO_{\Line^1\times\Line^2}(2,2)|^{sm}$ of $|\ShO_{\Line^1\times\Line^2}(2,2)|^{ss}$ parametrizes smooth $(2,2)$-divisors on $\Line^1\times\Line^2$. The closed set $|\ShO_{\Line^1\times\Line^2}(2,2)|^{ss}\setminus |\ShO_{\Line^1\times\Line^2}(2,2)|^{sm}$ is $\SL_2(k)\times \SL_3(k)$-invariant. 
Let $M^{sm}_{(2,2)}$ be the image of $|\ShO_{\Line^1\times\Line^2}(2,2)|^{sm}$ by the quotient map $|\ShO_{\Line^1\times\Line^2}(2,2)|^{ss}\to M_{(2,2)}$. Then, we have a good quotient $|\ShO_{\Line^1\times\Line^2}(2,2)|^{sm}\to M^{sm}_{(2,2)}$.

For $Z\in M^{sm}_{(2,2)}(k)$, let $\Psi(Z)$ be the isomorphism class of a double cover of $\Line^1\times \Line^2$ ramified along $Z$. 
\begin{prop}(\cite[Proposition 9.3]{DJK24} and \cite[Theorem 4.5 (i)]{13})\label{prop 4.19}
The map
\begin{align*}
\Psi:M^{sm}_{(2,2)}(k)\to
\{\text{smooth Fano threefolds of \textnumero 2.18}\}/\cong
\end{align*}
is bijective. 
\end{prop}
\begin{proof}
It is enough to show that $\Psi$ is injective. Let $X_i$, with $i=1,2$, be double covers of $\Line^1\times \Line^2$ ramified along a smooth $(2,2)$-divisor $Z_i$. Assume that we have an isomorphism $f:X_1\to X_2$. By Proposition \ref{two dc}, there exists an automorphism $\sigma$ on $\Line^1\times \Line^2$ such that the following diagram is commutative:
\[
  \begin{CD}
     X_1 @>>> \Line^1\times\Line^2 \\
  @V{f}VV    @V{\sigma}VV \\
      X_2  @>>>  \Line^1\times\Line^2.
  \end{CD}
\]
Here, the horizontal arrows are the double covers. The automorphism $\sigma$ is represented by $M\in \SL_2(k)$ and $N\in \SL_3(k)$, and we obtain the equation $M\!\times\! N(Z_1)=\sigma(Z_1)=Z_2$. Then, we see that $Z_1$ and $Z_2$ are equal in $M^{sm}_{(2,2)}(k)$. 
\end{proof}

By the above proposition, the quasi-projective variety $M^{sm}_{(2,2)}$ parametrizes the isomorphism classes of smooth Fano threefolds of \textnumero 2.18.

Following \cite[Section 2.7]{DJK24}, we introduce the GIT moduli space of plane quartic curves $|\ShO_{\Line^2}(4)|^{ss}/\!\!/ \SL_3$. The reductive group $\SL_3$ naturally regularly acts on $|\ShO_{\Line^2}(4)|$. Thus, we get the GIT-quotient \[M_4:=|\ShO_{\Line^2}(4)|^{ss}/\!\!/ \SL_3.\]
For a $(2,2)$-divisor 
\[Z=\{t_0^2Q_1+2t_0t_1Q_2+t_1^2Q_3=0\}\subset \Line^1_{[t_0:t_1]}\times\Line^2_{[x:y:z]}\]
on $\Line^1\times\Line^2$, we get the plane quartic curve 
\[\Delta=\{Q_2^2-Q_1Q_3=0\}\subset \Line^2_{[x:y:z]}.\] 
Thus we obtain the rational map
\[\Phi: M_{(2,2)}\dashrightarrow M_4; Z\mapsto \Delta\]
(see \cite[Section 1.2]{DJK24} and \cite[Theorem 4.5 (i)]{13}).

The stable locus $|\ShO_{\Line^2}(4)|^{s}$ consists of plane quartic curves with at worst only $A_1$ or $A_2$ singularities. In addition, for any smooth $(2,2)$-divisor $Z$, the plane quartic curve $\Phi(Z)$ has at worst $A_1$-singularities. Therefore the rational map $\Phi$ is defined on $M^{sm}_{(2,2)}$. By \cite[Section 1.2 and 1.3]{DJK24}, we see that $\Phi$ is dominant. Indeed, the following lemma holds:

\begin{lem}\label{morph form (2,2) to 4}(\cite[Theorem 4.5 (i)]{13}, \cite[Section 3]{Bru08} and \cite[Section 1.2 and 1.3]{DJK24})
The morphism \[\Phi:M^{sm}_{(2,2)}\to M_4\] is dominant. 
\end{lem}

\begin{proof}
Smooth plane quartic curves form an open locus $M^{sm}_4$ of $M_4$. It is enough to show that the map
\[\Phi:\Phi^{-1}(M^{sm}_4)\to M^{sm}_4\]
is surjective. For any smooth plane quartic curve $\Delta\subset\Line^2_{[x:y:z]}$, there exists an \'etale double covering $\bar{\omega}:\tilde{\Delta}\to \Delta$. By \cite[Theorem 4.5]{13}, there exist quadric forms $Q_1,Q_2,Q_3\in k[x,y,z]$ such that $\tilde{\Delta}$ and $\Delta$ are given by
\[\tilde{\Delta}=V_+(Q_1-r^2,Q_2-rs,Q_3-s^2)\subset \Line^4_{[x:y:z:r:s]}\,\,\text{and}\, \Delta=V_+(Q_2^2-Q_1Q_3)\subset \Line^2_{[x:y:z]},\]
respectively. By Proposition \ref{smoothness}, the $(2,2)$-surface
\[Z:t_0^2Q_1+2t_0t_1Q_2+t_1^2Q_3=0\subset \Line^1_{[t_0:t_1]}\times \Line^2_{[x:y:z]}\]
is smooth, and we have $\Phi(Z)=\Delta$.
\end{proof}

\begin{thm}\label{aut of general}
Let $\pi:X\to \Line^1\times\Line^2$ be a double cover ramified along a general $(2,2)$-divisor. Then the automorphism group of $X$ is the cyclic group of order 2, generated by the non-trivial involution acting trivially on $\Line^1\times\Line^2$. 

Therefore, for a general smooth Fano threefold $X\in M^{sm}_{(2,2)}(k)$ of \textnumero 2.18, we get $\Aut(X)\cong \Z/2\Z$.
\end{thm}

\begin{proof}
Since a general plane quartic curve is a smooth plane quartic whose automorphism group is trivial, this theorem follows from Lemma \ref{morph form (2,2) to 4} and Corollary \ref{upper of aut of 2.18}.
\end{proof}

\section{Linearizability}\label{section 5}
In this section, we work over an algebraically closed field of characteristic zero. We introduce an application of the equivariant IJT-obstruction. In addition, we apply the calculation of automorphism groups given in Section 4 to linearizability. The main objects are smooth Fano threefolds of \textnumero 2.18. 

First, we discuss equivariant polarized Prym schemes and equivariant Prym varieties. To construct examples of abelian varieties that are not a product of Jacobian varieties, Mumford introduced Prym varieties. To apply the rationality problem for geometrically rational conic bundles, \cite{13} constructs polarized Prym schemes. \cite{1} gives equivariant versions of these, called equivariant polarized Prym schemes and equivariant Prym varieties, and these equivariant versions apply to the linearizability problem for conic bundles.

Theorem \ref{CTT24 Theorem 4.4}, Proposition \ref{CTT24 prop 5.1}, Theorem \ref{CTT24 thm 5.2}, Theorem \ref{CTT24 Theorem 5.3}, Proposition \ref{CTT24 prop 5.6}, Theorem \ref{CTT24 Theorem 5.4}, Theorem \ref{CTT24 Example 5.14} and Proposition \ref{non-linearizable prop} can be proven in the same way as \cite{1}.

\subsection{Equivariant polarized Prym schemes and Prym varieties}\label{section 5.1}
Let $G$ be a finite group and  
\[\bar{\omega}:\tilde{\Delta}\to\Delta\]
be a $G$-equivariant \'etale double cover between smooth projective irreducible $G$-curves. Then the norm map
\[\Nm(\bar{\omega}):\PPic_{\tilde{\Delta}/k}\to \PPic_{\Delta/k}\]
is a $G$-equivariant homomorphism of $G$-group schemes. 

Let $r:\Delta\inj W$ be a $G$-equivariant embedding into a smooth rational projective surface $W$ with a regular $G$-action. Then $r$ induces a $G$-equivariant homomorphism of group schemes
\[r^*:\PPic_{W/k}\to \PPic_{\Delta/k}.\]
The $G$-equivariant $\PPic_{W/k}$-polarized Prym scheme of the $G$-equivariant \'etale double cover $\bar{\omega}$ is defined as
\[\PPrym_{\tilde{\Delta}/\Delta}^{\PPic_{W/k}}:=\PPic_{\tilde{\Delta}/k}\times_{\PPic_{\Delta/k}}\PPic_{W/k}, \]
which is a group scheme with diagonal $G$-action. 
In addition, the $G$-equivariant Prym variety $\Prym_{\tilde{\Delta}/\Delta}$ is defined by
\[\Prym_{\tilde{\Delta}/\Delta}:=\left(\PPrym_{\tilde{\Delta}/\Delta}^{\PPic_{W/k}}\right)^0=(\ker(\Nm(\bar{\omega})))^0,\]
which is a $G$-abelian variety. The restriction of a $G$-invariant principal polarization associated with a theta divisor on $\PPic_{\tilde{\Delta}/k}$ to $\Prym_{\tilde{\Delta}/\Delta}$ is twice a $G$-equivariant principal polarization on $\Prym_{\tilde{\Delta}/\Delta}$. Therefore, we see that $\Prym_{\tilde{\Delta}/\Delta}$ is a $G$-equivariant principally polarized abelian variety.

For each $D\in \Pic(W)$, the fiber of the second projection 
\[\PPrym_{\tilde{\Delta}/\Delta}^{\PPic_{W/k}}\to \PPic_{W/k}\]
over $D$ is denoted by  $V_D$. 

For any $D\in \Pic(W)$, the fiber $V_D$ over $D$ is only dependent on the class $r^*D$. In particular, if $r^*$ is injective, for example when $W\Gcong \Line^2$, we have the $G$-equivariant isomorphism 
\[\PPrym_{\tilde{\Delta}/\Delta}^{\PPic_{W/k}}(k)\Gcong\left\{D\in \Pic(\tilde{\Delta})\middle|\bar{\omega}_*(D)\in r^*(\Pic(W))\right\}\]
via the first projection.

If $D$ is $G$-invariant, $V_D$ is a $G$-equivariant torsor of $V_0$. 
Because we have $\bar{\omega}_*\bar{\omega}^*D=2D$, we get 
\[V_{2D+D'}=\bar{\omega}^*r^*D+V_{D'}\]
as subschemes of $\PPrym_{\tilde{\Delta}/\Delta}^{\PPic_{W/k}}$ for any $D,D'\in \Pic(W)$.

The group scheme $V_0$ consists of two connected components,
\[P^{(0)}=P:=\Prym_{\tilde{\Delta}/\Delta},\quad \tilde{P}^{(0)}=\tilde{P}:=V_0\setminus P.\]
Both carry the $G$-action, and $\tilde{P}$ is a $G$-equivariant torsor of $P$.

\begin{prop}(\cite[Section 6, Equation (6.1)]{Mum74})
Let $\Delta\subset \Line^2$ be a smooth plane quartic and $H$ be the hyperplane class of $\Line^2$. Then, the parity of $h^0$ is constant on each of the two connected components of $V_H$, and is different on these components. Here $h^0(X)=H^0(X,\ShO_X)$. 
\end{prop}

\begin{dfn}(\cite[Definition 4.4]{13} and \cite[Definition 4.2]{1})
Let $\Delta\subset \Line^2$ be a smooth plane quartic, $H$ be the hyperplane class of $\Line^2$, and $\tilde{\Delta}$ be a \'etale double cover of $\Delta$. Assume that a finite group $G$ regularly acts on $\Delta$ and $\tilde{\Delta}$, and the double cover $\tilde{\Delta}\to \Delta$ is $G$-equivariant. Note that a regular $G$-action on $\Delta$ induces a regular $G$-action $\Line^2$ such that the immersion $\Delta\inj \Line^2$ is $G$-equivariant. 
\begin{itemize}
\item[(1)]
The connected components of $V_H$ on which $h^0$ is even and odd are denoted by $P^{(1)}$ and $\tilde{P}^{(1)}$, respectively.
\item[(2)]
For each $m\in \Z_{\ge 0}$ and $e=0,1$, we define 
\[P^{(2m+e)}:=P^{(e)}+mH, \,\tilde{P}^{(2m+e)}:=\tilde{P}^{(e)}+mH.\]
Then these are $G$-equivariant tosors of $P=\Prym_{\tilde{\Delta}/\Delta}$.
\item[(3)]
For each $m\in \Z_{> 0}$, we define the subschemes $S^{(m)}$ and $\tilde{S}^{(m)}$ of $\Sym^m\tilde{\Delta}$ to be the preimages of $P^{(m)}$ and $\tilde{P}^{(m)}$, respectively, under the Abel-Jacobi map $\Sym^m\tilde{\Delta}\to \PPic^m_{\tilde{\Delta}/k}$.
\end{itemize}
\end{dfn}

\subsection{Intermediate Jacobians and Prym varieties}\label{section 5.2}
We give a relation between Chow group schemes and polarized Prym schemes. It induces an equivalence of intermediate Jacobians and Prym varieties.

Let $X$ and $W$ be a smooth projective rational threefold and a surface with a regular action of a finite group $G$, respectively. Assume the $G$-equivariant morphism $\pi:X\to W$ is a standard conic bundle on $W$ over $k$. Since the $G$-action on $W$ preserves $\Delta$, the finite group $G$ regularly acts on $\Delta$. Assume that the discriminant curve $\Delta\subset W$ is smooth and the $G$-equivariant \'etale cover $\bar{\omega}:\tilde{\Delta}\to\Delta$ is irreducible, i.e., $\tilde{\Delta}$ is irreducible, hence so is $\Delta$.

By the construction, the morphisms given by \cite[Theorem 5.1]{13} are $G$-equivariant. 
\begin{thm}(\cite[Theorem 4.4]{1} and \cite[Theorem 5.1 and 5.8]{13})\label{CTT24 Theorem 4.4}
There exists a $G$-equivariant surjective morphism 
\[\CCH^2_{X/k}\to \PPrym^{\PPic_{W/k}}_{\tilde{\Delta}/\Delta}\]
of group schemes and its restriction to the identity component is a $G$-equivariant isomorphism 
\[(\CCH^2_{X/k})^0\Gcong \Prym_{\tilde{\Delta}/\Delta}\]
as $G$-equivariant principally polarized abelian varieties.
\end{thm}

\begin{proof}
We give an outline of the proof. We have a homomorphism 
\[\varphi:\CCH^2_{X/k}\to \PPic_{\tilde{\Delta}/k}.\]
The morphism $\varphi$ is given in the proof of \cite[Theorem 5.1 (iii)]{13}. By construction of $\varphi$, the morphism $\varphi$ is $G$-equivariant. Since $\pi$ is $G$-equivariant, the homomorphism 
\[\pi_*: \CCH^2_{X/k}\to  \PPic_{W/k}\]
so is. Hence the homomorphism
\begin{equation}\label{aaaaa}
(\varphi,\pi_*):\CCH^2_{X/k}\to \PPic_{\tilde{\Delta}/k}\times \PPic_{W/k}
\end{equation}
is $G$-equivariant for the diagonal $G$-action on $\PPic_{\tilde{\Delta}/k}\times \PPic_{W/k}$. The morphism \eqref{aaaaa} is the morphism given in the proof of \cite[Theorem 5.1]{13}. Then we get a $G$-equivariant morphism 
\[\CCH^2_{X/k}\to \PPrym^{\PPic_{W/k}}_{\tilde{\Delta}/\Delta}\]
of group schemes. The $G$-equivariant morphism is the morphism given in \cite[Theorem 5.7]{13}. By \cite[Theorem 5.1 and 5.8]{13}, the assertions hold.
\end{proof}

\subsection{Linearizability of Fano threefolds of \textnumero 2.18}\label{section 5.3}
Following \cite{1}, we introduce tools for the linearizability of Fano threefolds of \textnumero 2.18. 
In addition, we give new examples of linearizable and non-linearizable rational threefolds.

In this section, let
\[\pi:X\to\Line^1\times\Line^2\] 
be the double cover branched in a smooth $(2,2)$-divisor over algebraically closed filed $k$ of characteristic $0$. Let
\[\pi_1:X\to\Line^1 \,\text{and}\, \pi_2:X\to \Line^2\]
be the compositions of $\pi$ and the first projection, and $\pi$ and the second projection, respectively. Assume that the discriminant curve $\Delta$ is smooth. Then we have an \'etale cover $\bar{\omega}:\tilde{\Delta}\to\Delta$ induced by the standard conic bundle $\pi_2$. Suppose that a finite group $G$ regularly acts on $X$. Then the $G$-action on $X$ induces regular $G$-actions on $\Delta, \Line^2, \Line^1$, and $\tilde{\Delta}$. In addition, the morphisms $\pi_1,\pi_2, \bar{\omega}$, and the immersion $\Delta\inj \Line^2$ are $G$-equivariant. Moreover, the $G$-action on $X$ induces a regular $G$-action on the Fano variety $\calF_1(X/\Line^1)$ of lines in the fibers of $\pi_1$. We have a $G$-equivariant Stein factorization
\[\calF_1(X/\Line^1)\to C\to \Line^1. \] 
Then $C$ is a smooth projective $G$-curve of genus $2$. (See \cite[Theorem 4.5 (ii) and Theorem 6.3 (iii)]{13}.)

\subsubsection*{Examples of linearizable actions}
\begin{prop}(\cite[Proposition 5.1]{1})\label{CTT24 prop 5.1}
Assume that $G$ is cyclic and $\tilde{\Delta}^G\neq \emptyset$.
Then the $G$-action on $X$ is linearizable. 
\end{prop}

\begin{proof}
We give an outline of the proof.
For $P\in \tilde{\Delta}^G(k)$, the point $\bar{\omega}(P)$ is a fixed point of $\Delta$. Let $\ell, \ell'$ be the lines in the fiber $\ell\cup \ell'$ of $\pi_2$ over $\bar{\omega}(P)$. Then $\ell$ and $\ell'$ are $G$-stable. Let $p:\hat{X}\to X$ be the $G$-equivariant blow-up along $\ell$, with the exceptional divisor $E$. 

Let $H_1$ and $H_2$ be the pull-backs of hyperplane classes by $\pi_1$ and $\pi_2$, respectively. The line bundle $L:=p^*H_1+p^*H_2-E$ is $G$-stable. As the linear system $|L|$ is $4$-dimensional and base point free, the associated morphism 
\[\Phi_{|L|}:\hat{X}\to \Line^4\]
induces a $G$-equivariant birational morphism to a quadric threefold $Q\subset \Line^4$. 

Then $Q$ contains a $G$-stable plane. Since $G$ is cyclic, by taking an eigenvalue, we see that $Q$ has a fixed point. Projection from the fixed point induces a $G$-equivariant birational map from $Q$ to $\Line^3$. Therefore, we get a $G$-equivariant birational map from $X$ to $\Line^3$. 
\end{proof}

\begin{cor}
Assume that $G$ is a cyclic group of odd order. Then the $G$-action on $X$ is linearizable. 
\end{cor}

\begin{proof}
Let $f\in k[x,y,z]$ be a homogeneous polynomial which satisfies $\Delta=V_+(f)$. 
By \cite[Proposition 14]{B08} or \cite[Lemma 6.5.1]{CAG}, choosing coordinates, the pair of the type of the generator of $G$ and $f$ is one of the following lists:\\
\begin{tabular}{|l|r|r|} \hline
   $ $&Type &$f(x,y,z)$\\ \hline \hline
(i)&3;(0,1)&$z^3L_1(x,y)+L_4(x,y)$\\\hline
(ii)&3;(1,2)  &$ x^4+\alpha x^2yz+xy^3+xz^3+\beta y^2z^2 $\\\hline
 (iii) &7;(3,1)  &$x^3y+y^3z+z^3x  $\\\hline
 (iv)  &9;(3,2)  &$ x^4+xy^3+z^3y $\\\hline
\end{tabular}\\

Here $L_i$ is a general homogeneous polynomial of degree $i$ in $k[x,y]$, and $\alpha, \beta \in k$. 
For any case, the point $[0:0:1]\in \Delta$ is a $G$-fixed point. 
The preimage $\bar{\omega}^{-1}([0:0:1])$ consists of distinct two points $P_1,P_2$. Since $\bar{\omega}$ is $G$-equivariant and $[0:0:1]$ is fixed point, $G$ acts on $\bar{\omega}^{-1}([0:0:1])=\{P_1,P_2\}$. Suppose that $P_1$, $P_2$ are not fixed by the $G$-action. Then the $G$-action on $\bar{\omega}^{-1}([0:0:1])$ is not trivial. Hence the order of $G$ is even. This contradicts the assumption that the order of $G$ is odd. Therefore we have a $G$-fixed point $P_i$.
The assertion follows from Proposition \ref{CTT24 prop 5.1}.
\end{proof}

\begin{example}
Let quadrics $Q_1,Q_2$ and $Q_3$ be $\sqrt{-1}(y^2+xz), x^2+ayz$ and $\sqrt{-1}(z^2+xy)$, respectively. Assume that $Q_2^2-Q_1Q_3$ is of the form
\[x^4+x(y^3+z^3)+\alpha yzx^2+\beta y^2z^2,\]
and $\Delta=V_+(Q_2^2-Q_1Q_3)$ is smooth.
Here $\alpha\neq \beta$ and $\alpha\beta\neq 0$ hold.
The $C_3$-action on 
\[X:w^2=Q_1t_0^2+2t_0t_1Q_2+Q_3t_1^2\]
is defined by 
\[([t_0:t_1],[x:y:z],w)\mapsto ([ t_0:\zeta_3^2t_1],[x:\zeta_3 y:\zeta_3^2 z],\zeta_3 w).\]
Then the $C_3$-action is linearizable.
\end{example}

\begin{thm}(\cite[Theorem 5.2]{1} and \cite[Theorem 4.5]{13})\label{CTT24 thm 5.2}
We have $G$-equivariant isomorphisms $\PPic^0_{C/k}\overset{\sim}{\to}\Prym_{\tilde{\Delta}/\Delta}$ and $\PPic^1_{C/k}\overset{\sim}{\to} P^{(1)}$. The first and second isomorphisms are isomorphisms as $G$-equivariant principally polarized abelian varieties and as $G$-equivariant torsors of $\PPic^0_{C/k}\cong\Prym_{\tilde{\Delta}/\Delta}$, respectively. 
\end{thm}

\begin{proof}
We give an outline of the proof. Let $x$ be a general point of $C$ with the image $t\in \Line^1$. Then the fiber $S_t$ by $\pi_1$ at $t$ is a smooth quadric surface. A general point $X$ corresponds to a ruling in of lines on $S_t$. By taking $\ell$ as a general in this ruling, we have an effective $0$-cycle $\pi_2^*\Delta\cap\ell$ of degree $4$, and its push-forward to define an effective $0$-cycle on $\tilde{\Delta}$ of degree $4$. Hence we get a $G$-equivariant morphism $C\to \PPic_{\tilde{\Delta}/k}$, and its image is contained in $P^{(1)}$. See \cite[Theorem 5.2]{1} for details. The morphism $C\to P^{(1)}$ induces a $G$-equivariant morphism $\PPic^1_{C/k}\to P^{(1)}$.
Since the $G$-equivariant morphism $\PPic^1_{C/k}\to P^{(1)}$ is the same morphism given in \cite[Theorem 4.5]{13} (See \cite[proposition 6.3 (iii)]{13}). Hence this morphism is isomorphic. By \cite[Section 5, Case 4]{Bru08}, the morphism $C\to P^{(1)}$ induces a $G$-equivariant isomorphism $\PPic^0_{C/k}\to P^{(0)}$.
\end{proof}

Let $\gamma_1$ and $\gamma_2$ be the classes of lines in the fibers of $\pi_1:X\to \Line^1$ and $\pi_2:X\to \Line^2$, respectively. Then the classes $\gamma_1$ and $\gamma_2$ generate $\NS^2(X)$ as a $\Z$-module.
\begin{thm}(\cite[Theorem 5.3]{1} and \cite[Theorem 6.4]{13})\label{CTT24 Theorem 5.3}
The $G$-equivariant isomorphism of group schemes in Theorem \ref{CTT24 Theorem 4.4} induces the following $G$-equivariant isomorphisms between connected components of each other 
\[(\CCH^2_{X/k})^{m\gamma_1+n\gamma_2}\Gcong 
\begin{cases}
P&\text{$m,n$ are even;}\\
\tilde{P}&\text{$m$ is even and $n$ is odd;}\\
P^{(1)}&\text{$m$ is odd and $n$ is even;}\\
\tilde{P}^{(1)}&\text{$m,n$ are odd.}
\end{cases}
\]
\end{thm}
\begin{proof}
The isomorphisms \cite[Theorem 6.4]{13} are induced by the $G$-equivariant morphism in \ref{CTT24 Theorem 4.4}. By \cite[Theorem 6.4]{13} and \ref{CTT24 thm 5.2}, we obtain these $G$-equivariant isomorphisms. 
\end{proof}
\begin{dfn}
We define 
\[\calF_{1,1}(X):=\overline{\calM}_{0.0}(X,\gamma_1+\gamma_2)\]
which is the coarse moduli of stable maps of genus $0$ and class $\gamma_1+\gamma_2$.
\end{dfn}
Similarly to \cite{1}, the following propositions hold.
\begin{prop}(\cite[Proposition 5.6]{1})\label{CTT24 prop 5.6}
The moduli space $\calF_{1,1}$ is a smooth projective threefold and the Abel-Jacobi map
\[\AJ:\calF_{1,1}(X)\to (\CCH^2_{X/k})^{\gamma_1+\gamma_2}\]
is a $\Line^1$-fibration. 
\end{prop}

\begin{proof}
We give an outline of the proof. 
Let $[f:R\to X]\in \calF_{1,1}(X)$ be a stable map of genus $0$. Then $f$ is one of the followings:
\begin{itemize}
\item[(1)] The curve $R$ is isomorphic to $\Line^1$, the map $f$ is an embedding, and the class of $f(R)$ is $\gamma_1+\gamma_2$. 
\item[(2)] The curve $R$ is an union of $R_1\cong \Line^1$ and $R_2\cong \Line^1$, and $f$ induces an isomorphism from $R_i$ to a line in the fiber of $\pi_i$, for $i=1,2$. 
\end{itemize}
In both types, we have $H^1(R,f^*T_X)=0$. Hence, we see that $[f:R\to X]\in \calF_{1,1}(X)$ is a smooth point. Therefore $\calF_{1,1}(X)$ is a smooth projective variety of dimension $-K_X.R=3$. 

Let $f:R\to X$ be a stable map of genus $0$. Then $\ell:=\pi_2(f(R))$ is a line in $\Line^2$. Hence we get a morphism 
\[\varphi: \calF_{1,1}(X)\to (\Line^2)^{\lor}.\]
We denote the Stein factorization of $\varphi$ by 
\[\calF_{1,1}(X)\to B\to(\Line^2)^{\lor}. \]
Then $\calF_{1,1}(X)\to B$ is a $\Line^1$-fibration, and the Abel-Jacobi map $\AJ$ is dominant with connected fibers. Therefore $\AJ$ is a $\Line^1$-fibration. See \cite[Proposition 5.6]{1} for details. 
\end{proof}

\begin{thm}(\cite[Theorem 5.4]{1})\label{CTT24 Theorem 5.4}
Assume that $G$ is cyclic. Then $X$ is $G$-linearizable if and only if $\tilde{P}$ or $\tilde{P}^{(1)}$ is a trivial torsor.
\end{thm}

\begin{proof}
Assume that $X$ is $G$-linearizable. By Theorem \ref{IJT}
, \ref{CTT24 thm 5.2}, 
and \ref{CTT24 Theorem 5.3}, one of $\tilde{P}$ and $\tilde{P}^{(1)}$ is isomorphic to $P$, and the other is isomorphic to $P^{(1)}$. Since $P$ is trivial, then $\tilde{P}$ or $\tilde{P}^{(1)}$ is a trivial torsor. See \cite[Proposition 5.7 and 5.8]{1} for the converse. 
\end{proof}

The following example is an application of the above theorem.
\begin{thm}\cite[Example 5.14]{1}\label{CTT24 Example 5.14}
Let $\tau$ be the non-trivial involution on $X$ acting trivially on $\Line^1\times\Line^2$. Then, $X$ is $\langle\tau\>$-linearizable. 
\end{thm}
\begin{proof}
We give an outline of the proof.
We find a curve $R$ of the class $\gamma_1+\gamma_2$ which induces a $\langle\tau\>$-fixed point of $\tilde{P}^{(1)}$. See \cite[Example 5.14]{1} for details. $\tilde{P}^{(1)}$ is a trivial $P$-torsor by the fixed point. 
\end{proof}

From the above theorem and Theorem \ref{aut of general}, we get the following corollary:
\begin{cor}\label{general linearizable}
A smooth Fano threefold X of \textnumero 2.18, corresponding to a general point of $M^{sm}_{(2,2)}$, is $\Aut(X)$-linearizable.
\end{cor}

\subsubsection*{Examples of non-linearizable actions}
Similar to \cite{1}, the following also holds. 
\begin{prop}(\cite[Proposition 5.9]{1})\label{non-linearizable prop}
Let $\tau\in \Aut(X)$ be an involution such that 
\begin{itemize}
\item[(1)] the induced action on $\Line^1$ via $\pi_1:X\to \Line^1$ is trivial, 
\item[(2)] the induced action on $\Line^2$ via $\pi_2:X\to \Line^2$ is non-trivial, and 
\item[(3)] $\tilde{\Delta}^{\tau}=\emptyset$.
\end{itemize}
Then $X$ is not $\langle \tau \rangle$-projectively linearizable. 
\end{prop}

\begin{proof}
We prove the claim following the method of \cite{1}. Let $Q\in k[x,y,z]$ be a homogeneous polynomial, which defines the discriminant curve $\Delta$. We may assume that the induced action on $\Line^2$ is given by
\[[x:y:z]\mapsto [x:y:-z],\]
and \[Q=z^4+L_2(x,y)z^2+L_4(x,y).\]
Here $L_2,L_4\in k[x,y]$ are homogeneous polynomials of degree $2$ and $4$, respectively. Because $\Delta$ is smooth, the equation $L_4=0$ has $4$ distinct roots.

Suppose that $X$ is $\Z/2\Z$-projectively linearizable. By Theorem \ref{IJT} and Theorem \ref{CTT24 thm 5.2}, there exists a smooth connected projective curve $C$ of genus $2$ with a regular $\Z/2\Z$-action such that we have a $\Z/2\Z$-equivariant isomorphism
\[(\CCH^2_{X/k})^{\gamma_1+\gamma_2}\cong\PPic^m_{C/k}, \]
for some $m$. 

The quotient map $C\to C/\langle\tau\>$ is a double cover. By Riemann-Hurwitz formula, we have the inequality 
\[2(g(C/\langle\tau\>)-1)\le g(C)-1=1.\]
Hence, we see that $g(C/\langle\tau\>)=0$ or $1$. Moreover, we see that the double cover $C\to C/\langle\tau\>$ is not \'etale. Therefore, the curve $C$ has a fixed point $P$.
Hence we get the $\Z/2\Z$-equivariant isomorphism
\[\PPic^0_{C/k}\to \PPic^m_{C/k}\]
given by $[D]\mapsto [mP+D]$. Because $\PPic^0_{C/k}$ has a fixed point, the $\Z/2\Z$-variety $(\CCH^2_{X/k})^{\gamma_1+\gamma_2}$ so does. By Proposition \ref{CTT24 prop 5.6}, the threefold $\calF_{1,1}(X)$ has a fixed point. 

Let $f:R\to X$ be a stable map which is a fixed point of $\calF_{1,1}(X)$. Then $f$ is $\langle\tau\>$-stable. 

If $R$ is reducible, the curve $R$ consists of two lines $L_1\cong \Line^1$ and $L_2\cong \Line^1$. Then one of $L_i$'s is isomorphically mapping to a line $\ell_1$ in a fiber of $\pi_1$ and the other $L_j$ to a line $\ell_2$ in a fiber of $\pi_2$ (see \cite[Proposition 5.6]{1}). We may assume that $f(L_i)=\ell_i$ for $i=1,2$. By the assumption (3), the curve $\tau(\ell_2)$ is a line, which is not equal to $\ell_2$, in a fiber of $\pi_2$. This contradicts the fact that $f$ is $\langle\tau\>$-stable.

Thus $R$ is irreducible and $R_1=f(R)$ is $\langle\tau\>$-stable. Then we get $R\cong \Line^1$, and $f:R\to X$ is an immersion whose image is of class $\gamma_1+\gamma_2$ (see \cite[Proposition 5.6]{1}). If $R_1$ is not fixed, we have a bisection of $\pi_1$ by the assumption (1). This contradicts the fact that the class of $R_1$ is $\gamma_1+\gamma_2$. We see that $R_1$ is fixed. Then the line $\pi_2(R_1)$ is fixed. Hence, we get $\pi_2(R_1)=\{z=0\}$. Since $L_4=0$ has distinct roots, the surface $S=\pi_2^{-1}(R_1)$ is a $G$-equivariant double cover of $\Line^1\times \ell$ branched over $4$ distinct lines.

Let $R'$ be the image of $R_1$ by the double cover $S\to \Line^1\times \ell$. Then we have a birational morphism $R_1\to R'$. Hence, the pull-back of $R'$ via the double cover consists of two curves $R_1$ and $R_2$. Since the double cover $S\to \Line^1\times \ell$ is $\langle\tau\>$-equivariant, we have $\tau(R_1)=R_2$. This contradicts the fact that $R_1$ is fixed.
\end{proof}

\begin{example}\label{nolinearaction}
Let $Q_1=\sqrt{-1}x^2+y^2, Q_2=z^2$ and $Q_3=\sqrt{-1}x^2-y^2$. 
Consider the automorphism
\[\tau:([t_0:t_1], [x:y:z],w)\mapsto \left(\left[\sqrt{-1}t_0:\frac{t_1}{\sqrt{-1}}\right], [\sqrt{-1}x:\sqrt{-1}y:z],w\right).\] Then $\langle \tau \rangle$ is a cyclic group of order $4$. The $\Z/4\Z$-action on $\Delta$, induced by $\pi_2$, has a fixed point, e.g. $[1:\zeta_8:0]$. Moreover, lines in a fiber above a $\Z/4\Z$-fixed point are $\Z/4\Z$-stable. Therefore $\langle \tau \rangle$ is linearizable. 

If $\sigma$ is the involution acting trivially on $\Line^1\times \Line^2$, the involution $\sigma \tau^2$ satisfies the assumptions in Proposition \ref{non-linearizable prop}. Hence the group generated by $\tau, \sigma$ is not projectively linearizable. 
\end{example}

G{\footnotesize RADUATE} S{\footnotesize CHOOL} {\footnotesize OF} M{\footnotesize ATHEMATICS}, N{\footnotesize AGOYA} U{\footnotesize NIVERSITY}, F{\footnotesize UROCHO} C{\footnotesize HIKUSAKU}, N{\footnotesize AGOYA}, 464-8602, J{\footnotesize APAN}\\
{\em Email address:} $\mathtt{shuto.abe.c6\text{@}math.nagoya\text{-}u.ac.jp}$


\begin{thebibliography}{24}

\bibitem{Shu25}
S. Abe,
\newblock {\em An obstruction of linearizability of threefold in characteristic zero,} Master's thesis. Nagoya University, 2025.


\bibitem{5}
D. Abramovich and M. Temkin,
\newblock {\em Functorial factorization of birational maps for qe schemes in characteristic 0,} Algebra Number Theory \textbf{13} (2019), no. 2, 379--424.

\bibitem{B08}
F. Bars, 
\newblock{\em Automorphisms groups of genus 3 curves, } Number Theory Seminar UAB-UB-UPC on Genus 3 curves, Barcelona, 2005.


\bibitem{BS83}
S. Bloch and V. Srinivas, 
\newblock {\em Remarks on correspondences and algebraic cycles,} Amer. J. Math. \textbf{105} (1983), no. 5, 1235--1253.


\bibitem{B79}
S. Bloch, 
\newblock {\em Torsion algebraic cycles and a theorem of Roitman,} Compos. Math. \textbf{39} (1979), no. 1, 107--127.


\bibitem{Bru08}
N. Bruin, 
\newblock {\em The arithmetic of Prym varieties in genus 3,} Compos. Math. \textbf{144} (2008), no. 2, 317--338.

\bibitem{3}
O. Benoist and O. Wittenberg,
 \newblock{\em The Clemens-Griffiths method over non-closed field,} Algebr. Geom. \textbf{7} (2020), no. 6, 696--721. 

\bibitem{11}
\underline{\,\,\,\,\,\,\,\,\,\,\,\,\,\,\,\,\,\,\,},
\newblock{\em Intermediate Jacobians and rationality over arbitrary fields.} Ann. Sci \'{E}c Norm. Sup\'{e}r \textbf{56} (2023), no. 4, 1029--1084.






\bibitem{CAG}
I. Dolgachev, 
\newblock {\em Classical Algebraic Geometry: a modern view,} First Edition, Cambridge University Press, 2012.

\bibitem{CFK23}
I. Cheltsov, K. Fujita, T. Kishimoto and J. Park, 
\newblock {\em K-stable Fano 3-folds in the families 2.18 and 3.4,} arXiv:2304.11334, 2023.



\bibitem{CMTZ24}
I. Cheltsov, L. Marquand, Y. Tschinkel and Z. Zhang, 
\newblock {\em Equivariant geometry
of singular cubic threefolds I\hspace{-1.2pt}I}, J. Lond. Math. Soc. (2) \textbf{112} (2025), no. 1, Paper No. e70224, 46.


\bibitem{1}
T. Ciurca, S. Tanimoto and Y. Tschinkel,
\newblock {\em Intermediate jacobians and linearizability,} arXiv:2403.06047, 2024, Kyoto J. of Math., to appear.

\bibitem{CTT} 
\underline{\,\,\,\,\,\,\,\,\,\,\,\,\,\,\,\,\,\,\,},
\newblock {\em EQUIVARIANT GEOMETRY OF CONIC BUNDLES,} Unpublished note.

\bibitem{CTZ25a}
I. Cheltsov, Y. Tschinkel and Z. Zhang, 
\newblock {\em Equivariant geometry of singular cubic threefolds,} Forum Math. Sigma \textbf{13} (2025), Paper No. e9, 52.

\bibitem{CTZ25b}
\underline{\,\,\,\,\,\,\,\,\,\,\,\,\,\,\,\,\,\,\,},
\newblock {\em Equivariant unirationality of Fano threefolds}, arXiv:2502.19598, 2025.



\bibitem{CTZ25c}
\underline{\,\,\,\,\,\,\,\,\,\,\,\,\,\,\,\,\,\,\,}, 
\newblock {\em Equivariant geometry of the Segre cubic and the Burkhardt quartic,} Selecta Math. (N.S.) \textbf{31} (2025), no. 1, Paper No. 7, 36.







\bibitem{12}
C. Clemens and P. Griffiths,
\newblock {\em The intermediate Jacobian of the cubic threefold,} Ann. of Math. (2) \textbf{95} (1972), 281--356.













\bibitem{DI09}
I. V. Dolgachev and V. A. Iskovskikh. {\em Finite subgroups of the plane Cremona group}, In {\em Algebra, arithmetic, and geometry: in honor of {Y}u. {I}. {M}anin. {V}ol. {I}}, Progr. Math. \textbf{269} (2009), 443--548. 



\bibitem{DJK24}
K. DeVleming, L. Ji, P. Kennedy-Hunt and M. H. Quek, 
\newblock {\em The K-moduli space of a family of conic bundle threefolds,} arXiv:2403.09557, 2024.

\bibitem{dFE02}
T. Fernex and L. Ein, {\em Resolution of indeterminacy of pairs}, In Algebraic geometry: A volume in memory of Paolo Francia (2003), 165--178.

\bibitem{13}
S. Frei, L. Ji, S. Sankar, B. Viray and I. Vogt,
\newblock {\em Curve classes on conic bundle threefolds and applications to rationality,} Algebr. Geom. \textbf{11} (2024), no. 3, 421--459. 



\bibitem{17}
W. Fulton,
\newblock {\em Intersection theory,} \textbf{2},  Springer-Verlag, Berlin, Second edition, 1998.



\bibitem{GW10}
U. Goertz and T. Wedhorn,
\newblock {\em Algebraic Geometry: Part I: Schemes. With Examples and Exercises,}  Springer Studium Mathematik - Master, Advanced Lectures in Mathematics, 2010. 



\bibitem{15}
B. Hassett and Y. Tschinkel,
\newblock {\em Cycle class maps and birational invariants,} Comm. Pure Appl. Math. \textbf{74} (2021), no. 12, 2675--2698. 

\bibitem{16}
\underline{\,\,\,\,\,\,\,\,\,\,\,\,\,\,\,\,\,\,\,},
\newblock {\em Rationality of complete intersections of two quadrics over nonclosed fields.} Enseign. Math. \textbf{67} (2021), no. 1-2, 1--44. With an appendix by Jean-Louis Colliot-Th\'{e}l\`{e}ne.

\bibitem{HT22}
\underline{\,\,\,\,\,\,\,\,\,\,\,\,\,\,\,\,\,\,\,},
\newblock {\em Equivariant geometry of odd-dimensional complete intersections of two quadrics,} Pure Appl. Math. Q. \textbf{18} (2022), no. 4, 1555--1597.


\bibitem{Hui79}
C.-M. Hui, 
\newblock {\em Plane quartic curves,}  PhD thesis, University of Liverpool, 1979.


\bibitem{IP} 
V. Iskovskikh and Yu. Prokhorov,
\newblock {\em Fano varieties,} Algebraic Geometry V, Encyclopaedia Math. Sci. \textbf{47}, Springer-Verlag, Berlin, 1999.

\bibitem{2}
R. Hartshorne,
\newblock {\em Algebraic Geometry,} Graduate Texts in Mathematics \textbf{52}, Springer, 1977.









\bibitem{Joe24}
J. Malbon, {\em Automorphisms of Fano threefolds of rank 2 and degree 28,} Ann. Univ. Ferrara Sez. VII Sci. Mat. \textbf{70} (2024), no. 3, 1083--1092. 

\bibitem{KM98}
J. Koll\'{a}r and S. Mori,
\newblock {\em Birational Geometry of Algebraic Varieties}, Cambridge University Press, 1998.

\bibitem{K24}
N. Konovalov,
\newblock {\em On the automorphism groups of smooth Fano threefolds}, arXiv:2406.03584, 2024.




\bibitem{KPS18}
A. G. Kuznetsov, Y. G. Prokhorov and C. A. Shramov, 
\newblock {\em Hilbert schemes of lines and conics and automorphism groups of Fano threefolds,} Jpn. J. Math. \textbf{13} (2018), no. 1, 109--185.




\bibitem{6}
K. Matsuki, 
\newblock {\em Lecture on factorization of birational maps,} RIMS preprint math, AG/0002084, arXiv:0002084, 2000.





\bibitem{MM81}
S. Mori and S. Mukai,
\newblock {\em Classification of Fano 3-folds with $B_2\ge 2$,} Manuscripta mathematica \textbf{36} (1981), no. 2, 147--162.

\bibitem{MM83}
\underline{\,\,\,\,\,\,\,\,\,\,\,\,\,\,\,\,\,\,\,}, 
\newblock {\em On Fano 3-folds with $B_2\ge 2$,} In {\em Algebraic varieties and analytic varieties (Tokyo, 1981)}, Adv. Stud. Pure Math. \textbf{1} (1983), 101--129.

\bibitem{MM03}
\underline{\,\,\,\,\,\,\,\,\,\,\,\,\,\,\,\,\,\,\,}, 
\newblock {\em Erratum: "Classification of Fano 3-folds with $B_2\ge 2$",} [Manuscripta Math. 36 (1981/82), no. 2, 147--162; MR0641971 (83f:14032)]. Manuscripta mathematica. \textbf{110} (2003), no. 3, 407.


\bibitem{Mue23}
Peter Mueller and TCiur, {\em Writing a smooth plane quartic as the vanishing of $Q_0Q_2-Q_1^2$ for quadratic $Q_0,Q_1,Q_3$ (question and answer) }, mathoverflow, 2023. \\ https://mathoverflow.net/questions/445951/writing-a-smooth-plane-quartic-as-the-vanishing-of-q-0q-2-q-12-for-quadrati.



\bibitem{9}
D. Mumford, 
\newblock {\em Abelian Varieties.} 3rd edition, Republished by the Tata Institute of Fundamental Research and distributed by the American Math Society, 2010.




\bibitem{Mum74}
\underline{\,\,\,\,\,\,\,\,\,\,\,\,\,\,\,\,\,\,\,}, 
\newblock {\em Prym varieties I,} In Contributions to analysis (a collection of papers dedicated to Lipman Bers), (1974), 325--350.

\bibitem{4}
J. P. Murre,
\newblock{\em Reduction of the proof of the non-rationality of a non-singular cubic threefold to a result of Mumford,} Compositio Math. \textbf{27} (1973), 63--82.


\bibitem{14}
\underline{\,\,\,\,\,\,\,\,\,\,\,\,\,\,\,\,\,\,\,}, 
\newblock {\em Applications of algebraic K-theory to the theory of algebraic cycles, } In {\em Algebraic geometry, Sitges (Barcelona), 1983}, Lecture Notes in Math. \textbf{1124} (1985), 216--261. 

\bibitem{PCS19}
V. V. Przyjalkowski, I. A. Cheltsov and K. A. Shramov,
\newblock {\em Fano threefolds with infinite automorphism groups,}  Izv. Ross. Akad. Nauk Ser. Mat. \textbf{83} (2019), no. 4, 226--280. 





\bibitem{PR20}
A.J. Parameswaran and N. Ray,
\newblock {\em Stability and semi-stability of $(2,2)$-type surfaces, }arXiv:2009.06232, 2020.

\bibitem{PSY24}
A. Pinardin, A. Sarikyan and E. Yasinsky,
\newblock {\em Linearization problem for finite subgroups of the plane Cremona group, }arXiv:2412.12022, 2024.


\bibitem{Pro13}
Y. G. Prokhorov,
\newblock {\em On birational involutions of $\Line^3$,} Izv. Ross. Akad. Nauk Ser. Mat. \textbf{77} (2013), no. 3, 199--222.




\bibitem{Pro18}
\underline{\,\,\,\,\,\,\,\,\,\,\,\,\,\,\,\,\,\,\,}, 
\newblock {\em The rationality problem for conic bundles,} Uspekhi Mat. Nauk \textbf{73} (2018), no. 3, 375--456.

\bibitem{Sar82}
V. G. Sarkisov, 
\newblock {\em On conic bundle structures,} Izv. Akad. Nauk SSSR Ser. Mat. \textbf{46} (1982), no. 2, 371--408, 432. 


\bibitem{TYZ23}
Y. Tschinkel, K. Yang and Z. Zhang,
\newblock {\em Equivariant birational geometry of linear actions,} EMS Surv. Math. Sci. \textbf{11} (2024), no. 2, 235--276. 

\bibitem{vBDL25}
R. van Bommel, J. Docking, R. Lercier and E. L. Garcia,
\newblock {\em Reduction of Plane Quartics and Dixmier-Ohno invariants,} Res. Number Theory \textbf{11} (2025), no. 1, Paper No. 41, 25.




\end{thebibliography}
\end{document}